\documentclass[11pt,reqno]{amsart}
\usepackage[a4paper, hmargin={2.7cm,2.7cm},vmargin={3.3cm,3.3cm}]{geometry}
\usepackage[english]{babel}
\usepackage{color}
\usepackage[noadjust]{cite}
\usepackage{amsmath}
\usepackage{amssymb}
\usepackage{amsfonts}
\usepackage{amsthm}
\usepackage{enumerate}
\usepackage{amsthm}
\usepackage{dsfont}
\usepackage{xcolor}

\usepackage[textwidth=20mm]{todonotes}
\usepackage[hypertexnames=false]{hyperref}
\usepackage{zref-clever}
\zcsetup{capfirst}
\usepackage{keytheorems}
\usepackage{commath}
\usepackage{esint}
\usepackage{etoolbox}

\usepackage{amssymb}   
\usepackage{amsthm}    
\usepackage{thmtools}  
\usepackage{mathtools} 
\usepackage{mathrsfs}  
\usepackage{commath}
\usepackage{bbm}
\usepackage[textwidth=20mm]{todonotes}

\numberwithin{equation}{section}

\newkeytheorem{theorem}[
  name=Theorem,
  parent=section,
  refname={theorem,theorems},
  Refname={Theorem,Theorems}
]

\newkeytheorem{lemma}[
  name=Lemma,
  sibling=theorem,
  refname={lemma,lemmaer},
  Refname={Lemma,Lemmaer}
]

\newkeytheorem{proposition}[
  name=Proposition,
  sibling=theorem,
  refname={proposition,propositions},
  Refname={Proposition,Propositions}
]

\newkeytheorem{corollary}[
  name=Corollary,
  sibling=theorem,
  refname={corollary,corollaries},
  Refname={Corollary,Corollaries}
]

\newkeytheorem{definition}[
  name=Definiton,
  style=definition,
  refname={definition,definitions},
  Refname={Definition,definitions}
]

\newkeytheorem{example}[
  name=Example,
  style=definition,
  refname={example,examples},
  Refname={Example,}
]

\usepackage[all]{xy}

\newcommand{\N}{\mathbb{N}}
\newcommand{\Z}{\mathbb{Z}}

\newcommand{\R}{\mathbb{R}}
\newcommand{\C}{\mathbb{C}}

\newcommand\numberthis{\addtocounter{equation}{1}\tag{\theequation}}

\newcommand{\Hrho}{\mathcal{H}}
\newcommand{\Hpi}{\mathcal{H}}

\DeclareMathOperator{\Pf}{Pf}

\DeclareMathOperator{\ind}{ind}

\DeclareMathOperator{\Span}{span}
\DeclareMathOperator*{\esssup}{ess\,sup}

\newcommand{\anal}{\mathsf{C}}
\newcommand{\synt}{\mathsf{D}}
\newcommand{\fram}{\mathsf{S}}
\newcommand{\gram}{\mathsf{G}}

\def\XXint#1#2#3{{\setbox0=\hbox{$#1{#2#3}{\int}$ }
\vcenter{\hbox{$#2#3$ }}\kern-.6\wd0}}

\title[A Balian-Low theorem for nilpotent Lie groups]{A Balian-Low theorem for nilpotent Lie groups}

\author{Ulrik Enstad}
\address{Department of Mathematics,
University of Oslo,
Moltke Moes vei 35,
0851 Oslo.}
\email{ubenstad@math.uio.no}

\author{Jordy Timo van Velthoven}
\address{Faculty of Mathematics,
University of Vienna, 
Oskar-Morgenstern-Platz 1,
1090 Vienna, Austria}
\email{jordy-timo.van-velthoven@univie.ac.at}



\begin{document}

\maketitle

\begin{abstract}
We prove strict necessary density conditions for smooth coherent frames and Riesz sequences in the orbit of a connected, simply connected, nilpotent Lie group. This extends previous results for homogeneous groups to general nilpotent Lie groups. 
The main novelty in our approach is that we prove a stability result for coherent frames and Riesz sequences based on deforming their representation rather than their index set.
\end{abstract}

\section{Introduction}
Let $N$ be a connected, simply connected, nilpotent Lie group and let $(\rho, \Hrho)$ be an irreducible representation of $N$ that is square-integrable modulo the center $Z(N)$. For a vector $\eta \in \Hrho$ and a discrete set $\Lambda \subseteq N/Z(N)$, the associated \emph{coherent system} is given by the family of vectors
\[
\rho(\Lambda) \eta = \big( \rho(x_{\lambda}) \eta \big)_{ \lambda \in \Lambda },
\]
where $(x_{\lambda} )_{\lambda \in \Lambda}$ are representatives in $N$ for the cosets $\Lambda \subseteq N / Z(N)$.\footnote{The vectors defining the coherent system $\rho(\Lambda) \eta$ do depend on the choice of representatives, but  their frame or Riesz sequence property is independent of this choice.}
The construction of such coherent systems forming a frame, Riesz sequence or orthonormal basis has been studied in, e.g., \cite{bedos2022smooth, oussa2024orthonormal, grochenig2018orthonormal, oussa2019compactly, enstad2022sufficient, bekka2004square}. Moreover, various results relating the spanning and reproducing properties of coherent systems and the density of their index set have been obtained in, e.g., \cite{papageorgiou2025counting, bedos2022smooth, fuehr2017density, grochenig2020balian}. See also \cite{velthoven2024density} for a survey on various of the aforementioned results.

Among the current key results on coherent frames and Riesz sequences are the necessary density conditions in terms of the Beurling density of their index set. Explicitly, the following density theorem was obtained in \cite{fuehr2017density, caspers2023overcompleteness}; see  \zcref{sec:coherent} for the relevant definitions. 

\begin{theorem}\label{prop:necessary_density_conditions}
Let $N$ be a connected, simply connected, nonabelian, nilpotent Lie group and let $(\rho, \Hrho)$ be an irreducible representation of $N$ that is square-integrable modulo the center $Z(N)$ of formal degree $d_{\rho} > 0$. Let $\Lambda \subseteq N/Z(N)$ be countable and let $\eta \in \Hrho$ be such that $\langle \eta, \rho (\cdot) \eta \rangle \in W(L^2)$. 
The following assertions hold:
\begin{enumerate}[(i)]
    \item If $\rho (\Lambda)\eta$ is a frame for $\Hrho$, then $D^-(\Lambda) \geq d_{\rho}$.
    \item If $\rho (\Lambda)\eta$ is a Riesz sequence in $\Hrho$, then $D^+(\Lambda) \leq d_{\rho}$.
\end{enumerate}
In particular, if $\rho(\Lambda) \eta$ is a Riesz basis, then $D^-(\Lambda) = D^+(\Lambda) = d_{\rho}$.
\end{theorem}

\zcref{prop:necessary_density_conditions} shows that the formal dimension $d_{\rho}$ is the critical density separating frames and Riesz sequences. The result forms an extension of the classical density theorem for Gabor systems \cite{heil2007history}, which correspond to the Heisenberg group. In the latter case, the density conditions in \zcref{prop:necessary_density_conditions} are additionally known to be \emph{strict} inequalities under additional smoothness conditions on $\eta$, i.e., when $\eta$ is a Schwartz function, cf. \cite{grochenig2015deformation, ascensi2014dilation}. These strict necessary density conditions provide a far reaching generalization of the classical Balian-Low theorem  \cite{battle1988heisenberg}, which asserts, in particular, that the system $\big \{ e^{2\pi i m \cdot } \eta (\cdot - n) : m,n \in \Z \big\}$ for a Schwartz function $\eta \in  \mathcal{S} (\R)$ cannot form an orthonormal basis for $L^2 (\mathbb{R})$.

Although a Balian-Low type theorem has been expected to hold for arbitrary connected, simply connected, nilpotent Lie groups (see, e.g., \cite{grochenig2018orthonormal, grochenig2020balian, velthoven2024density}), it has so far only been known for the subclass of so-called \emph{homogeneous groups}, i.e., nilpotent Lie groups that admit automorphism dilations. More precisely, it was shown in \cite[Theorem 1.1]{grochenig2020balian} that for homogeneous groups the necessary density inequalities in \zcref{prop:necessary_density_conditions} must be \emph{strict} whenever $\eta$ is a smooth vector for $\rho$; in notation, $\eta \in \Hrho^{\infty}$. Nevertheless, there exist numerous nonhomogeneous nilpotent Lie groups admitting square-integrable representations, see, e.g., \cite[Theorem 1.1]{burde2025characteristically} and \cite[Example 7.5]{beltita2015transference}.

The aim of the present paper is to prove a Balian-Low theorem for general nilpotent Lie groups. The following theorem is our main result.

\begin{theorem} \label{thm:main_intro}
With notation as in \zcref{prop:necessary_density_conditions}, the following assertions hold:
\begin{enumerate}[(i)]
    \item If $\rho(\Lambda) \eta$ is a frame for $\Hrho$ with $\eta \in \Hrho^{\infty}$, then $D^-(\Lambda) > d_{\rho}$.
    \item If $\rho(\Lambda) \eta$ is a Riesz sequence in $\Hrho$ with $\eta \in \Hrho^{\infty}$, then $D^+ (\Lambda) < d_{\rho}$.
\end{enumerate}
In particular, if $\rho(\Lambda) \eta$ is a Riesz basis for $\Hrho$, then $\eta \notin \Hrho^{\infty}$.
\end{theorem}

\zcref{thm:main_intro} solves an open problem that was implicitly posed in \cite{grochenig2018orthonormal, grochenig2020balian} and that was explicitly formulated as \cite[Question 12]{velthoven2024density}. Moreover, it shows the optimality of the frame and Riesz sequence constructions in \cite{bownik2025frame, bownik2026riesz}, which prove that for every $\varepsilon > 0$ and $\eta \in \Hrho^{\infty} \setminus \{0\}$, there exists $\Lambda \subseteq N/Z(N)$ with density $D^+ (\Lambda) \leq (1+\varepsilon) d_{\rho}$ (resp. $D^- (\Lambda) \geq (1-\varepsilon) d_{\rho}$) such that $\rho(\Lambda) \eta$ is a frame (resp. Riesz sequence) for $\Hrho$.

Our overall proof approach for \zcref{thm:main_intro} is based on a classical line of arguments that argues through a deformation result for frames and Riesz sequences. Specifically, we will show that if $\rho(\Lambda) \eta$ forms a frame (resp. Riesz sequence) for $\Hrho$, then there exists some $t>1$ (resp. $t < 1$) such that an associated coherent system $\rho_t (\Lambda) \eta$ is also a frame (resp. Riesz sequence) for $\Hrho$ with formal dimension $d_{\rho_t} = t^d d_{\rho}$ for some $d \in \mathbb{N}$, which implies via \zcref{prop:necessary_density_conditions} that $D^- (\Lambda) \geq d_{\rho_t} > d_{\rho}$ (resp. $D^+ (\Lambda) < d_{\rho}$). A similar overall deformation argument underlies the Balian-Low type theorems \cite{grochenig2015deformation, ascensi2014dilation, grochenig2020balian} in the case there exist automorphism dilations $(\delta_t)_{t > 0}$ on $G$. In that case, it was shown in \cite{grochenig2015deformation, ascensi2014dilation, grochenig2020balian} that if $\rho(\Lambda) \eta$ is a frame (resp. Riesz sequence) for $\Hrho$, then so is $\rho(\delta_t (\Lambda)) \eta$ for some $t>1$ (resp. $t<1$). The key difference between the approach in the present paper and the papers \cite{grochenig2015deformation, ascensi2014dilation, grochenig2020balian} is, however, that we deform the representation from $\rho = \rho_1$ to some $\rho_t$ and leave the index set $\Lambda$ fixed, whereas the papers  \cite{grochenig2015deformation, ascensi2014dilation, grochenig2020balian} change the index set from $\Lambda = \delta_1(\Lambda)$ to $\delta_t (\Lambda)$ while leaving the representation $\rho$ fixed. Although both types of deformations are closely related in the setting of homogeneous groups, they differ significantly beyond that setting. Specifically, in our setting, $\rho_t$ is still a representation, whereas for a more general deformation $(D_t)_{t > 0}$, the maps $\rho \circ D_t$ are typically no longer representations.

For proving our deformation theorem for frames and Riesz sequences, we use the Kirillov theory for representations of nilpotent Lie groups  to deform the  representation $\rho$. Specifically, we use that any irreducible representation of a nilpotent Lie group is induced from a character associated to a functional on the Lie algebra. By scaling this functional by $t>0$, we therefore obtain a family $(\rho_t)_{t>0}$ of irreducible representations $\rho_t$ of $N$ that are square-integrable modulo $Z(N)$. For any square-integrable representation, it is well-known that the matrix coefficients associated to smooth vectors form Schwartz functions on the quotient group $N/Z(N)$, so that the family $(\rho_t)_{t>0}$ provides us with a family of Schwartz functions on $N/Z(N)$. 
Our main result on the deformation of representations (\zcref{thm:continuity-in-t}) is that these matrix coefficients are continuous in $t>0$ in the Schwartz space topology. Using this deformation theorem for matrix coefficients, we then show that certain frame operators (resp. Gramian matrices) associated to the systems $\rho_t (\Lambda)$ are continuous at $t=1$ in the operator norm topology (see \zcref{thm:common-deformation}), which implies the frame (resp. Riesz sequence) property of $\rho_t (\Lambda) \eta$ for all $t > 0$ sufficiently close to $1$.

The remainder of this paper is organized as follows. \zcref{sec:smooth0} provides basic background on nilpotent Lie groups and their representations. The deformation theorem for square-integrable representations is proven in \zcref{sec:deformation_representation}. Our main results on coherent systems, including \zcref{thm:main_intro}, are proven in \zcref{sec:coherent}. We prove a version of \zcref{thm:main_intro} for projective representations in \zcref{sec:projective}. Lastly, several auxiliary results are collected in the appendices.

\section{Smooth vectors for nilpotent Lie groups} \label{sec:smooth0}

Throughout this section, $N$ will denote a connected, simply connected, nilpotent Lie group with Lie algebra $\mathfrak{n}$ and exponential map $\exp_N : \mathfrak{n} \to N$. We let $\mu_N$ denote a left Haar measure on $N$. As a standard reference for representation theory of nilpotent Lie groups, we mention the book \cite{corwin1990representations}.

\subsection{Malcev bases} \label{sec:malcev}
Let $X_1, ..., X_n$ be a basis for $\mathfrak{n}$ and let $\mathfrak{n}_1 \subseteq \mathfrak{n}_2 \subseteq ... \subseteq \mathfrak{n}_k = \mathfrak{n}$ be subalgebras with $\dim \mathfrak{n}_j = n_j$. The basis $X_1, ..., X_n$ is called a \emph{weak Malcev basis} for $\mathfrak{n}$ passing through $\mathfrak{n}_1  \subseteq \cdots \subseteq \mathfrak{n}_k = \mathfrak{n}$ if it satisfies the conditions
\begin{enumerate}
    \item[(m1)] for each $1 \leq m \leq n$,  $\mathfrak{m}_m := \mathbb{R}$-$\Span \{ X_1, ..., X_m \}$ is a subalgebra of $\mathfrak{n}$,
    \item[(m2)] for each $1 \leq j \leq k$, $\mathfrak{m}_{n_j} = \mathfrak{n}_j$.
\end{enumerate}
In case the subalgebras $\mathfrak{n}_1 \subseteq \mathfrak{n}_2 \subseteq ... \subseteq \mathfrak{n}_k = \mathfrak{n}$ are ideals,  the basis $X_1, ..., X_n$ is said to be a \emph{strong Malcev basis} for $\mathfrak{n}$ passing through $\mathfrak{n}_1 \subseteq  ... \subseteq \mathfrak{n}_k = \mathfrak{n}$ if it satisfies (m2) and
\begin{enumerate}
    \item[(m3)] for each $1 \leq m \leq n$, $\mathfrak{m}_m := \mathbb{R}$-$\Span \{ X_1, ..., X_m \}$ is an ideal of $\mathfrak{n}$.
\end{enumerate}
In case we do not list any subalgebras (resp. ideals), we will simply refer to the basis as a \emph{weak} (resp. \emph{strong}) Malcev basis for $\mathfrak{n}$. See, e.g., \cite[Theorem 1.1.13]{corwin1990representations} for the existence of Malcev bases.

If $\mathfrak{m} \subseteq \mathfrak{n}$ is an $m$-dimensional subalgebra, $M := \exp_N (\mathfrak{m})$ and $X_1, ..., X_n$ is a weak Malcev basis for $\mathfrak{n}$ passing through $\mathfrak{m}$, then the map 
\[
\phi : \mathbb{R}^{n - m} \to N/M, \quad (t_{m+1}, ..., t_{n}) \mapsto \exp_N (t_n X_{n}) \cdots \exp_N (t_{m+1} X_{m+1}) M
\]
is a diffeomorphism. In particular, the map $s : N/M \to N$ given by
\[
s(\exp_N (t_n X_n) \cdots \exp_N (t_{m+1} X_{m+1}) M) = \exp_N (t_n X_n) \cdots \exp_N (t_{m+1} X_{m+1})
\]
is a smooth cross-section for the canonical quotient map $p : N \to N/M$, i.e., $p \circ s = \operatorname{id}_{N/M}$.
We sometimes refer to such a cross-section as a \emph{Malcev cross-section}.

\subsection{Induced and irreducible representations}
Let $\ell \in \mathfrak{n}^*$. A Lie subalgebra $\mathfrak m\subseteq\mathfrak n$ is said to be
\emph{subordinate} to $\ell$ if
$
    \ell([\mathfrak m,\mathfrak m])=0.
$
In this case,
\begin{align} \label{eq:character}
    \chi_\ell(\exp_N (X)) =e^{2\pi i\ell(X)},
    \qquad X\in\mathfrak m,
\end{align}
defines a unitary character of $M=\exp_N(\mathfrak m)$, and we write
$
    \rho(\ell,\mathfrak m)
    =
    \operatorname{Ind}_M^N(\chi_\ell) 
$ for the associated induced representation of $N$.

For $\ell\in\mathfrak n^*$, define its radical by
\[
    \mathfrak n(\ell)
    :=
    \big \{ X\in\mathfrak n:
      \ell([X,Y])=0
      \text{ for every }Y\in\mathfrak n \big\}.
\]
A subordinate subalgebra $\mathfrak m$ is called a \emph{polarization} for
$\ell$ if
\[
    \dim (\mathfrak m)
    =
    \frac{\dim (\mathfrak n ) +\dim (\mathfrak n(\ell))}{2}.
\]
If $\mathfrak m$ is a polarization for $\ell$, then
$\rho(\ell,\mathfrak m)$ is irreducible, and its unitary equivalence
class is independent of the choice of polarization. We will therefore often simply write $\rho_\ell = \rho(\ell,\mathfrak m)$. Note that in this case the central character of $\rho_{\ell}$, that is, the character $\chi$ satisfying $\rho_{\ell}|_{Z(N)} = \chi \cdot I_{\Hpi}$, is given by \eqref{eq:character}, because the center $\mathfrak{z}(\mathfrak{n})$ of $\mathfrak{n}$ is contained in the polarization $\mathfrak{m}$.

Conversely, by Kirillov's orbit method, every irreducible
unitary representation of $N$ is equivalent to an induced representation of the form $\rho_\ell$ for some
$\ell\in\mathfrak n^*$.

\subsection{Smooth vectors} \label{sec:smooth}
Let $(\rho, \Hrho)$ be an irreducible representation of $N$. A vector $\eta \in \Hrho$ is said to be \emph{smooth} if the orbit map $x \mapsto \rho(x) \eta$ is smooth on $N$. We denote the space of all smooth vectors by $\Hrho^{\infty}$. The Lie algebra $\mathfrak{n}$ acts on $\Hrho^{\infty}$ by means of the derived representation
\[
d\rho(X) \eta = \frac{d}{dt} \bigg|_{t=0} \rho(\exp_N (tX)) \eta, \quad X \in \mathfrak{n}, \; \eta \in \Hrho^{\infty}.
\]
Given a basis $X_1, ..., X_n$ of $\mathfrak{n}$, a family of semi-norms on 
$\Hrho^{\infty}$ can be defined as
\[
p_{\alpha} (\eta) := \big\| d\rho(X^{\alpha}) \eta \big\|_{\Hrho},
\quad \alpha \in \mathbb{N}_0^n.
\]
This gives $\Hrho^\infty$ the structure of a Fr\'echet space, and the group $N$ acts continuously on $\Hrho^{\infty}$.

For $\ell \in \mathfrak{n}^*$ and an irreducible representation $\rho_{\ell} := \ind^N_{\exp_N (\mathfrak{m})} (\chi_{\ell})$, one can use a weak Malcev basis of $\mathfrak{n}$ passing through the polarization $\mathfrak{m}$ to obtain a realization of $\rho_{\ell}$ in $L^2(\R^{n - m})$. Indeed, if $X_1, ..., X_n$ is such a weak Malcev basis, 
then the associated cross-section identifies $N/M$ with $\mathbb{R}^{n-m}$, and hence realizes the action of $\rho_{\ell}$ on $L^2 (\R^{n-m}) \cong  L^2(N/M)$. Following \cite[Section 4.1]{corwin1990representations}, we refer to such a realization of $\rho_{\ell}$ as a \emph{basis realization}. Note that a basis realization only depends on the choice of weak Malcev basis passing through $\mathfrak{m}$.

The following theorem is \cite[Corollary 4.1.2]{corwin1990representations}.

\begin{theorem}[\cite{corwin1990representations}] \label{thm:smooth_schwartz}
Let $\rho_{\ell}$ be an irreducible representation in the basis realization. Then its space of smooth vectors is given by $\mathcal{H}_{\rho_{\ell}}^{\infty} = \mathcal{S}(\R^{n-m})$ as Fr\'echet spaces. 
\end{theorem}

Lastly, we mention the compatibility of induction in stages and basis realizations. Specifically, let  $\mathfrak n_0 \subseteq \mathfrak n$ be a Lie subalgebra and suppose
that $\mathfrak m \subseteq \mathfrak n_0$ is a polarization for both
$\ell$ and $\ell_0 := \ell|_{\mathfrak n_0}$. Set
$N_0 := \exp_N(\mathfrak n_0)$,
$k := \dim(\mathfrak n)-\dim(\mathfrak n_0)$, and
$d := \dim(\mathfrak n)-\dim(\mathfrak m)$.
A weak Malcev basis of $\mathfrak n_0$ passing through $\mathfrak m$
can be extended to a weak Malcev basis of $\mathfrak n$.
Using these bases, induction in stages gives
$\rho_\ell \cong \operatorname{ind}_{N_0}^{N}\rho_{\ell_0}$
in the realization on
$L^2\bigl(\mathbb R^k;L^2(\mathbb R^{d-k})\bigr)$.
In this realization, the smooth vectors are given by
\[
    \mathcal H_{\rho_\ell}^{\infty}
    = \mathcal S(\mathbb R^d)
    \cong \mathcal S\bigl(\mathbb R^k;\mathcal S(\mathbb R^{d-k})\bigr),
\]
where the last identification is given by $f \mapsto \big(u \mapsto f(u,\cdot) \big)$.

\subsection{Square-integrable representations}
An irreducible representation $(\rho, \Hrho)$ of $N$ is called \emph{square-integrable modulo the center $Z(N)$} if there exists nonzero $f \in \Hrho$ such that
\[
\int_{N/Z(N)} |\langle f, \rho(x) f \rangle |^2 \; d\mu_{N/Z(N)} (x) < \infty,
\]
where integration is with respect to a fixed left Haar measure on $N/Z(N)$. In this case, there exists a constant $d_{\rho} > 0$ depending on the normalization of the Haar measure on $N/Z(N)$ such that
\begin{align} \label{eq:ortho_rho}
\int_{N/Z(N)} |\langle f_1, \rho(x) f_2 \rangle |^2 \; d \mu_{N/Z(N)} (x) = d_{\rho}^{-1} \| f_1 \|^2 \| f_2 \|^2
\end{align}
for all $f_1, f_2 \in \Hrho$. The identity \eqref{eq:ortho_rho} is referred to as the \emph{orthogonality relations} for $\rho$ and $d_{\rho} > 0$ is called the \emph{formal degree} of $\rho$.

An irreducible representation $\rho_{\ell}$ is square-integrable modulo $Z(N)$ if and only if the radical of $\ell \in \mathfrak{n}^*$ satisfies
\[
\mathfrak{n}(\ell) = \mathfrak{z}(\mathfrak{n}),
\]
where $\mathfrak{z}(\mathfrak{n})$ denotes the Lie algebra of $Z(N)$, see, e.g., \cite[Corollary 4.5.4]{corwin1990representations}. In this case, since $\mathfrak{n}(\ell)$ has even codimension (cf. \cite[Lemma 1.3.2]{corwin1990representations}), we have that $\dim(N/Z(N)) = 2d$ for some $d \in \mathbb{N}_0$. 

The following lemma provides an explicit expression for the formal degree of a square-integrable representation, see, e.g., \cite[Theorem 4]{moore1974square} and \cite[Theorem 4.5.10]{corwin1990representations}.

\begin{lemma}[\cite{moore1974square}] \label{lem:pfaffian}
Let $(\rho,\Hrho)$ be an irreducible unitary representation of $N$
that is square-integrable modulo $Z(N)$, and let $\ell\in\mathfrak n^*$
be such that $\rho=\rho_\ell$. Write
$
    \dim\bigl(N/Z(N)\bigr)=2d
$ for $d \in \N_0$
and fix a strong Malcev basis
$
    Z_1,\ldots,Z_r,X_1,\ldots,X_{2d}
$
of $\mathfrak n$ such that
$
    \mathfrak z(\mathfrak n)
    =
    \mathbb{R}$-$\Span \{Z_1,\ldots,Z_r\}.
$
Let $\Pf(\ell)$ denote the Pfaffian of the
 matrix $(\ell([X_i,X_j]))_{i,j=1}^{2d}$, so that 
\[
    \operatorname{Pf}(\ell)^2
    =\det\big(\ell([X_i,X_j])\big)_{i,j=1}^{2d}.
\]
Then there exists a constant $c>0$ such that
\[
    d_\rho
    =
    c\lvert\Pf(\ell)\rvert.
\]
The constant $c$ depends only on the choice of Malcev basis
and the normalization of the left Haar measure $\mu_{N/Z(N)}$.
\end{lemma}

\subsection{Schwartz matrix coefficients}  \label{sec:schwartz}
Let
$X_1,\ldots,X_n$ be a weak Malcev basis for $\mathfrak n$  passing through
$
    \mathfrak z(\mathfrak n)
    =
    \mathbb{R}$-$\Span \{X_1,\ldots,X_r\}
$ and let
$
    \phi\colon\mathbb R^{n-r} \to N/Z(N)
$
be the associated coordinate diffeomorphism. 
We define the associated Schwartz space on $N / Z(N)$ by
\[
    \mathcal S(N/Z(N))
    :=
    \bigl\{
        F\in C^\infty(N/Z(N)):
        F\circ\phi\in\mathcal S(\mathbb R^{n-r})
    \bigr\}.
\]
This definition is independent of the weak Malcev basis passing
through $\mathfrak z(\mathfrak n)$. It also agrees with the usual
definition of the Schwartz space on the nilpotent Lie group
$N/Z(N)$ obtained through its exponential map. See, e.g.,
\cite[p.~347]{howe1977on} for an explanation of both facts.

We will also use a characterization of the Schwartz space on $N/Z(N)$ in terms of seminorms. For this, we denote by $d$ a left-invariant Carnot-Carath\'eodory metric on $N/Z(N)$.
For $s \in \R$, we define a weight $w_s : N/Z(N) \to [0, \infty)$  by
\begin{align} \label{eq:weight}
    w_s(x) = (1 + d(x, e))^s, \quad x \in N/Z(N).
\end{align}
The Schwartz space $\mathcal{S}(N/Z(N))$ may then be described as
\begin{equation}
    \mathcal{S}(N/Z(N)) = \bigcap_{s > 0} \bigg\{ f \in C^\infty(N/Z(N)) \: : \:  \big\| (Df) \cdot w_s \big\|_\infty < \infty \; \text{for all} \; D \in \mathfrak{D} \bigg\} , \label{eq:schwartz-intersection}
\end{equation}
where $\mathfrak{D}$ denotes the algebra of all left-invariant differential operators on $N/Z(N)$, see, e.g., \cite{schweitzer1993dense, ludwig1995schwartz, ludwig1988minimal}.

Let $(\rho,\Hrho)$ be an irreducible representation of $N$ that is square-integrable modulo
$Z(N)$. For $f_1, f_2 \in\Hrho^\infty$, define the associated matrix coefficient on $N/Z(N)$ by
\[
    C^\rho_{f_1,f_2}(x)
    :=
    \bigl\langle f_1,\rho(s(x)) f_2 \bigr\rangle,
    \qquad x\in N/Z(N),
\]
where 
$
    s\colon N/Z(N) \to N
$
is a smooth Malcev cross-section for the projection $p : N \to N/Z(N)$.

It is well-known that the matrix coefficient associated to two smooth vectors is a Schwartz function on $N/Z(N)$. Specifically, the following theorem is known, see, e.g., \cite[Theorem 4.5.11]{corwin1990representations} and \cite[Theorem 1.2]{corwin1996lp}, or related results in \cite{howe1977on, pedersen1994matrix}.

\begin{theorem} \label{thm:schwartz_matrix}
Let $(\rho, \Hrho)$ be an irreducible representation of $N$ that is square-integrable modulo $Z(N)$. For every $f_1,f_2 \in\Hrho^\infty$, one has
$
    C^\rho_{f_1, f_2}\in\mathcal S(N/Z(N)).
$ 
Moreover, the map
\[
  C^{\rho} :  \Hrho^\infty\times\Hrho^\infty
    \longrightarrow
    \mathcal S(N/Z(N)),
    \qquad
    (f_1,f_2)\longmapsto C^\rho_{f_1, f_2},
\]
is continuous.
\end{theorem}

\section{Deformation of square-integrable representations} \label{sec:deformation_representation}
In this section, we continue the notation from \zcref{sec:smooth0} and let $(\rho, \Hrho)$ be an irreducible representation of $N$ that is square-integrable modulo the center $Z(N)$. Recall that then $\dim(N/Z(N)) = 2d$ for some $d \in \mathbb{N}_0$.

Let $\ell \in \mathfrak{n}^*$ be such that $\rho = \rho_{\ell}$. 
For $t > 0$, let $\rho_{t} := \rho_{t \ell}$ be the representation of $N$ associated to the functional $t \ell \in \mathfrak{n}^*$. Note first that
\[
\mathfrak{n}(t \ell) = \mathfrak{n}(\ell) = \mathfrak{z}(\mathfrak{n}), \quad t > 0.
\]
Moreover, if $\mathfrak{m}$ is a polarization for $\ell$, then it is also a polarization for $t \ell$ for any $t>0$. Indeed, $\ell([\mathfrak{m}, \mathfrak{m}]) = 0$ implies $(t\ell)([\mathfrak{m}, \mathfrak{m}]) = 0$ for $t > 0$, and the fact that $\mathfrak{n}(t\ell) = \mathfrak{z}(\mathfrak{n})$  yields that $\dim(\mathfrak{m}) = \frac{1}{2} ( \dim(\mathfrak{n}) + \dim(\mathfrak{n}(t \ell)))$ for $t > 0$, so that
\[ \dim(\mathfrak{n})-\dim(\mathfrak{m}) = \frac{1}{2} \dim(N/Z(N)) = d.\]
Thus, each $\rho_t$ is an irreducible representation of $N$ on $\Hpi$ that is square-integrable modulo $Z(N)$. Using \zcref{lem:pfaffian}, we have the following expression of the formal dimensions.

\begin{lemma}\label{lem:dim-scaled-rep}
The formal degrees $d_{\rho}$ and $d_{\rho_t}$ of $(\rho, \Hrho)$ and $(\rho_t, \mathcal{H}_{\rho_t})$, respectively, are related by
    \[ d_{\rho_t} = t^d d_{\rho}, \quad t > 0 . \]
\end{lemma}
\begin{proof}
     With respect to the fixed Haar measure and the fixed
basis used in \zcref{lem:pfaffian}, we have
$
    d_{\rho_t}=c\lvert\Pf(t\ell)\rvert,
$
where the constant $c>0$ is independent of $t$.
Since $ \mathfrak n/\mathfrak z(\mathfrak n)$ has dimension $2d$,
the Pfaffians corresponding to $t\ell, \ell \in \mathfrak{n}^*$ satisfy
\[
    \lvert\Pf(t\ell)\rvert
    =t^d \lvert\Pf(\ell)\rvert.
\]
Consequently, we have
$
    d_{\rho_t}
    =ct^d\lvert\Pf(\ell)\rvert
    =t^d d_\rho.
$
\end{proof}

For the family of representations $(\rho_t)_{t > 0}$, one can choose a fixed polarization and a common basis realization on  $L^2 (\R^{d})$.  By \zcref{thm:smooth_schwartz}, their smooth vectors are then given by
\[ \mathcal{H}_{\rho_t}^{\infty} = \mathcal{S}(\R^{d}), \quad t > 0. \]
 In addition, it follows from \zcref{thm:schwartz_matrix} that each matrix coefficient
\[
C^{\rho_t}_{f_1,f_2} \in \mathcal{S}(N/Z(N)), \quad t > 0, \;\; f_1, f_2 \in \mathcal{S}(\R^d).
\]

The following convergence theorem is the main result of this section. Our proof is modeled on the proof of \zcref{thm:schwartz_matrix} given in \cite{corwin1996lp}.

\begin{theorem}\label{thm:continuity-in-t}
 Let $(\rho,\Hrho)$ be an irreducible unitary representation of $N$
that is square-integrable modulo $Z(N)$, and let $\ell\in\mathfrak n^*$
be such that $\rho \simeq \rho_\ell$. 

There exists a common basis realization of the representations
$\rho_t:=\rho_{t\ell}$, $t>0$, on $L^2(\mathbb R^d)$
such that the coefficient maps
\[
    C^{\rho_t}\colon
    \mathcal S(\mathbb R^d)\times\mathcal S(\mathbb R^d)
    \longrightarrow \mathcal S(N/Z(N))
\]
depend continuously on $t\in(0,\infty)$ in the topology of
uniform convergence on products of bounded subsets of
$\mathcal S(\mathbb R^d)$.
\end{theorem}

\begin{proof}
    The proof is modeled on that of \cite[Theorem 1.2]{corwin1996lp} for showing \zcref{thm:schwartz_matrix}.

    Recall that $\dim(N/Z(N)) = 2d$ for some $d \in \mathbb{N}_0$. We prove the theorem by induction on $d$. For $d=0$ the theorem is immediate. Hence, we assume that $d\geq 1$ and that the theorem holds in any setting where the respective dimension is given by $2(d-1)$. We let $t_0 \in (0, \infty)$ and let $I \subseteq (0, \infty)$ be a compact interval whose interior contains $t_0$ and assume throughout that $t \in I$. Fix bounded sets $\mathcal{B}_1, \mathcal{B}_2 \subseteq \mathcal{S}(\R^d)$. We will prove convergence $C^{\rho_t}_{f_1, f_2} \to C^{\rho_{t_0}}_{f_1, f_2}$ uniformly for $f_1 \in \mathcal{B}_1$ and $ f_2 \in \mathcal{B}_2$.

  Since $\rho_\ell$ is
square-integrable modulo the center, we have that
$
\mathfrak n(\ell)=\mathfrak z(\mathfrak n)
$.  Define the normal subgroup
$
K :=\exp_N\bigl(\ker(\ell|_{\mathfrak z(\mathfrak n)})\bigr).
$
Then every $\rho_t$ factors through $N/K$, the center of $N/K$ is
$Z(N)/K$, and
$
(N/K)/(Z(N)/K)\cong N/Z(N).
$
Lifting a polarization and an adapted Malcev basis from $N/K$
to $N$ gives basis realizations with identical coefficient maps
under the above identification of quotient groups and Schwartz
spaces. As such, it suffices to prove the assertion for $N/K$,
whose center is one-dimensional. In what follows, we will therefore assume that $\dim(Z(N)) = 1$.

    Let $X_0 \in \mathfrak{z} (\mathfrak{n}) = \mathfrak{n}(\ell)$ be a nonzero element satisfying $\ell(X_0) = 1$, which is possible as $\mathfrak{n}$ is nonabelian.  
    By Kirillov's lemma \cite[Lemma~1.1.12]{corwin1990representations}, there exist
$X_1, X_{2d} \in \mathfrak{n}$ such that $[X_{2d}, X_1] = X_0$ 
and such that the centralizer $\mathfrak{c}_{\mathfrak{n}}(X_1) := \{X \in \mathfrak{n} : [X,X_1]=0\}$ is an ideal of
codimension one.
Replacing $X_1$ by $X_1-\ell(X_1)X_0$, we may assume that
$\ell(X_1)=0$.
Using a strong Malcev basis  for the chain of ideals
\[
    \mathbb{R}X_0
    \subseteq \mathbb{R}X_0 \oplus \mathbb{R}X_1
    \subseteq \mathfrak{c}_{\mathfrak{n}}(X_1)
    \subseteq \mathfrak{n},
\]
we obtain a strong Malcev basis
$X_0, X_1, \ldots, X_{2d}$ for $\mathfrak{n}$ satisfying
\begin{align} \label{eq:choice_basis}
    [X_1, X_j] = 0, \quad 0 \leq j < 2d, \quad
    \ell(X_1) = 0, \quad [X_{2d}, X_1] = X_0.
\end{align}
Lastly, we set
$\mathfrak{n}_0 := \mathfrak{c}_{\mathfrak{n}}(X_1)  = \mathbb{R}$-$\Span\{X_0, \ldots, X_{2d-1}\}$.

    The remainder of the proof is split into three steps.
    \\~\\
    \textbf{Step 1.} In this step, we provide a realization of the representations $\rho_t$ and associated matrix coefficients that is convenient for applying the induction hypothesis. 
    
    Set $N_0 := \exp_N (\mathfrak{n}_0)$ and let $\ell_0 := \ell |_{\mathfrak{n}_0}$. Given $t> 0$, let $\rho_{0,t}$ be the representation of $N_0$ associated with $t \ell_0$.
Let \(\mathfrak m_0\subseteq\mathfrak n_0\) be a polarization for
\(\ell_0\), and put
\(M_0=\exp_N(\mathfrak m_0)\).  Note that $\mathfrak{m}_0$ is also a polarization for $t \ell_0$ for every $t>0$. Moreover, the normalization $\ell(X_1) = 0$
together with
\(\mathfrak n(\ell)=\mathbb RX_0\), implies that
$
    \mathfrak z(\mathfrak n_0)
    =\mathfrak n_0(\ell_0)
    =\mathbb RX_0\oplus\mathbb RX_1
$, and thus
\[
    \dim(\mathfrak m_0)
    =\frac{\dim (\mathfrak n_0) +\dim (\mathfrak n_0(\ell_0))}{2}
    =\frac{2d+2}{2}
    =d+1.
\]
Since \(\mathfrak m_0\) is subordinate to \(t\ell\) and
\[
    \frac{\dim (\mathfrak n) +\dim\mathfrak n(t\ell)}{2}
    =\frac{(2d+1)+1}{2}
    =d+1,
\]
it is also a polarization for \(t\ell\).  Consequently, the
representations may be realized as
\[
    \rho_{0,t}
    =\ind_{M_0}^{N_0}\chi_{t\ell_0}
    \qquad \text{and} \qquad
    \rho_t
    =\ind_{M_0}^{N}\chi_{t\ell}.
\]
Since
$
    \mathfrak z(\mathfrak n_0)
    =\mathfrak n_0(\ell_0)
$, the representation $\rho_{0,t}$ is square-integrable representation modulo $Z(N_0)$. Moreover, since $\dim(N_0/Z(N_0))=2d-2$, the induction hypothesis
allows us to choose the polarization $\mathfrak m_0$ and
a common basis realization of the representations $\rho_{0,t}$
on $L^2(\mathbb R^{d-1})$ for which the conclusion of the theorem holds.
Using that \(\chi_{t\ell} =\chi_{t\ell_0}\), induction in stages gives the equivalence
\[
    \rho_t
    \simeq
    \ind_{N_0}^{N}\rho_{0,t}.
\]
Using this common realization of $\rho_{0,t}$, we realize
$\rho_t=\ind_{N_0}^{N}\rho_{0,t}$ on
$L^2(\mathbb R;L^2(\mathbb R^{d-1}))$.
In particular, its smooth vectors are given by
$
    \mathcal H_{\rho_t}^{\infty}
    =\mathcal S(\mathbb R^d)
    \cong\mathcal S\bigl(\mathbb R;
                        \mathcal S(\mathbb R^{d-1})\bigr),
$
see \zcref{sec:smooth}. 

We will next obtain the action of  $\rho_t = \ind_{N_0}^{N}\rho_{0,t}$
 and describe the associated matrix coefficients.  Since \(\mathfrak n_0\) is an ideal of codimension one, we can write
\[
    N=N_0\rtimes\exp_N(\mathbb RX_{2d}).
\]
We use the coordinates 
\[
    (z,y,w,x)
    =
    \exp_N(zX_0)\exp_N(yX_1)
    \exp_N\left(\sum_{j=2}^{2d-1}w_{j-1}X_j\right)
    \exp_N(xX_{2d}),
    \qquad w\in\R^{2d-2},
\]
and represent elements of \(N/Z(N)\) by setting \(z=0\).  Using the
cross-section \(u\mapsto\exp_N(-uX_{2d})\) of \(N/N_0\), the action of $\rho_t = \ind_{N_0}^N \rho_{0,t}$ is simply given by
\begin{align}
    (\rho_t(0,y,w,x)f)(u)
    =
    \rho_{0,t}\bigl(
       (0,0,0,u)(0,y,w,0)
       (0,0,0,-u)
    \bigr)f(x+u)
    \label{eq:induced-action}
\end{align}
for \(f \in L^2(\R;L^2(\R^{d-1}))\), see, e.g., \cite[Section 2.4, p. 79]{kaniuth2013induced} for a semidirect product realization.
In order to simplify the action \eqref{eq:induced-action} further,  first note that
\begin{equation}
    (0,0,0,u)(0,0,w,0)(0,0,0,-u)
    =
    \bigl(\alpha(w,u),\beta(w,u),\gamma(w,u),0\bigr).
    \label{eq:conjugation}
\end{equation}
The maps 
\(\alpha,\beta\), and \(\gamma\) are polynomial functions and, for fixed \(u\in\R\), the map
$
    w\longmapsto\gamma(w,u)
$
has a polynomial inverse, namely $\gamma(\gamma(w,-u),u)=w$.
Second, since $
    [X_{2d},X_1]=X_0
$ and $X_0$ is central, the conjugation formula for the exponential map yields
\begin{align} \label{eq:conjugate2}
    (0,0,0,u)(0,y,0,0)(0,0,0,-u)=(uy,y,0,0).
\end{align}
A combination of \eqref{eq:conjugation} and \eqref{eq:conjugate2} gives
\begin{align*}
    (0, 0, 0, u)(0,y,w,0) (0, 0, 0, -u)
    &=
    (uy,y,0,0)
    \bigl(\alpha(w,u),\beta(w,u),\gamma(w,u),0\bigr)\\
    &=
    \bigl(
       uy+\alpha(w,u),\,
       y+\beta(w,u),\,
       \gamma(w,u),\,
       0
    \bigr).
\end{align*}
Hence, since the central character of \(\rho_{0,t}\) on
\(\exp_N(\mathbb RX_0)\) is \(z\mapsto e^{2\pi itz}\), while
\(\rho_{0,t}\) is trivial on \(\exp_N(\mathbb RX_1)\), the action
\eqref{eq:induced-action} becomes
\begin{equation}
    (\rho_t(0,y,w,x)f)(u)
    =
    e^{2\pi it(uy+\alpha(w,u))}
    \rho_{0,t}(0,0,\gamma(w,u),0)f(x+u).
    \label{eq:explicit-induced-action}
\end{equation}
It follows that the associated matrix coefficient is given by
\begin{equation}
    C^{\rho_t}_{f_1,f_2}(y,w,x)
    =
    \int_{\R}e^{-2\pi ityu}
    A^t_{f_1,f_2}(w,x,u)\,du,
    \label{eq:coefficient-formula}
\end{equation}
where
\begin{equation}
    A^t_{f_1,f_2}(w,x,u)
    :=
    e^{-2\pi it\alpha(w,u)}
    C^{\rho_{0,t}}_{f_1(u),f_2(x+u)}
       \bigl(\gamma(w,u)\bigr).
    \label{eq:amplitude}
\end{equation}
We will refer to $A^t_{f_1,f_2}$ as the \emph{amplitude function} of the matrix coefficient $C^{\rho_t}_{f_1,f_2}$. 
\\~\\
\textbf{Step 2.} In this step, we will show that the amplitude functions 
\begin{align} \label{eq:amplitude_convergence}
    A^t_{f_1,f_2} \to A^{t_0}_{f_1, f_2} \quad \text{as} \quad t \to t_0 
\end{align} 
in $\mathcal{S}(\R^{2d})$, uniformly for $f_1 \in \mathcal{B}_1$ and $f_2 \in \mathcal{B}_2$.

In order to prove the claim \eqref{eq:amplitude_convergence}, we first show that the induction hypothesis implies that the functions $(v, r, s) \mapsto C^{\rho_{0,t}}_{f_1(r), f_2(s)} (v)$ converge to $(v, r, s) \mapsto C^{\rho_{0,t_0}}_{f_1(r), f_2(s)} (v)$ in $\mathcal{S}(\R^{2d})$, uniformly for $f_1 \in \mathcal{B}_1$ and $f_2 \in \mathcal{B}_2$. Under the identification
$
    \mathcal S(\mathbb R^{2d})
    =
    \mathcal S\bigl(
        \mathbb R^2;\mathcal S(\mathbb R^{2d-2})
    \bigr),
$
this means that, for every Schwartz seminorm $q$ on $\mathcal S(\mathbb R^{2d-2})$ and every $a,b,L\in\mathbb N_0$,
\begin{align}
\sup_{r,s\in\mathbb R}
 &(1+|r|+|s|)^L
 q\left(
 \partial_r^a\partial_s^b
 \left(
 C^{\rho_{0,t}}_{f_1(r),f_2(s)}
 -
 C^{\rho_{0,t_0}}_{f_1(r),f_2(s)}
 \right)
 \right)
 \to 0, \quad \text{as} \quad t \to t_0,
\label{eq:coefficient_vector_valued}
\end{align}
uniformly in $f_1\in\mathcal B_1$ and $f_2\in\mathcal B_2$.
For this, consider  $C^{\rho_{0,t}} : \mathcal{S}(\R^{d-1}) \times \mathcal{S}(\R^{d-1}) \to \mathcal{S}(\R^{2d -2})$ given by $(h_1, h_2) \mapsto C^{\rho_{0,t}}_{h_1, h_2}$. Note that $C^{\rho_{0,t}}$ is well-defined and continuous by \zcref{thm:schwartz_matrix}. Moreover, the induction hypothesis gives that
\[
C^{\rho_{0,t}} \to C^{\rho_{0,t_0}} \quad \text{as} \quad t \to t_0
\]
uniformly on products of bounded subsets of $\mathcal{S}(\R^{d-1})$. Note that under the canonical identification 
$\mathcal{S}(\R^d) = \mathcal{S}(\R, \mathcal{S}(\R^{d-1}))$ of a function $f \in \mathcal{S}(\R^d)$ with $f(r, \cdot) \in \mathcal{S}(\R^{d-1})$, 
the assumption that $\mathcal{B} \subseteq \mathcal{S}(\R^d)$ is bounded is equivalent to, for every $L, a \in \mathbb{N}_0$, the set
\begin{align} \label{eq:boundedset}
\big\{ (1+|r| )^L \partial_r^a f (r, \cdot) : f \in \mathcal{B}, r \in \mathbb{R} \big\}
\end{align}
being bounded in $\mathcal{S}(\R^{d-1})$. To ease notation, we write $f^{(a)} (r) := \partial_r^a f (r, \cdot)$. Then we have that
\[
\partial_r^{a} \partial_s^{b} C^{\rho_{0,t}}_{f_1 (r), f_2 (s)} (v) = C^{\rho_{0,t}}_{f_1^{(a)} (r), f_2^{(b)} (s)} (v), \quad v \in \R^{2d - 2}, \;\; r,s \in \R.
\]
Since
$
    (1+|r|+|s|)^L
    \leq
    (1+|r|)^L(1+|s|)^L,
$
the left-hand side of \eqref{eq:coefficient_vector_valued} is bounded by
\begin{align*}
\sup_{r,s\in\mathbb R}
 q\Big(
 &C^{\rho_{0,t}}_{
      (1+|r|)^Lf_1^{(a)}(r),
      (1+|s|)^Lf_2^{(b)}(s)}
 -
 C^{\rho_{0,t_0}}_{
      (1+|r|)^Lf_1^{(a)}(r),
      (1+|s|)^Lf_2^{(b)}(s)}
 \Big),
\end{align*}
which tends to zero as $t \to t_0$ by the induction hypothesis, as claimed.

We next use the convergence \eqref{eq:coefficient_vector_valued} to prove the claim \eqref{eq:amplitude_convergence} for the amplitude functions. Define the polynomial map $P : \R^{2d} \to \R^{2d}$ by
$
P(w,x,u) = \bigl(\gamma(w,u),u,x+u\bigr).
$
The identity $\gamma(\gamma(w,u),-u)=w$ gives that its inverse is given by
$
P^{-1}(v,r,s)=\bigl(\gamma(v,-r),s-r,r\bigr).
$
Thus, $P$ is a polynomial function with polynomial inverse. It follows therefore from \cite[p.~346]{howe1977on} that the map $F\mapsto F\circ P$ is a continuous automorphism of $\mathcal{S}(\R^{2d})$. Applying this map to the convergence \eqref{eq:coefficient_vector_valued} gives that
\begin{align} \label{eq:unmodulated_convergence}
C^{\rho_{0,t}}_{f_1(u),f_2(x+u)}\bigl(\gamma(w,u)\bigr)
\to
C^{\rho_{0,t_0}}_{f_1(u),f_2(x+u)}\bigl(\gamma(w,u)\bigr)
\end{align}
in $\mathcal{S}(\R^{2d})$, uniformly for $f_1 \in \mathcal{B}_1$ and $f_2 \in \mathcal{B}_2$.

We will next show that the set of functions
\begin{align} \label{eq:bounded_functions}
(w,x,u)\mapsto
C^{\rho_{0,t}}_{f_1(u),f_2(x+u)}
\bigl(\gamma(w,u)\bigr),
\qquad
t\in I,\quad f_1\in\mathcal B_1,\quad f_2\in\mathcal B_2,
\end{align}
forms a bounded subset of $\mathcal S(\R^{2d})$. 
Indeed, the induction
hypothesis states that $t\mapsto C^{\rho_{0,t}}$ is continuous in the
topology of uniform convergence on products of bounded subsets of
$\mathcal S(\R^{d-1})$. 
In particular, applying this to bounded sets of the form \eqref{eq:boundedset}, it follows that for
every Schwartz seminorm $q$ on $\mathcal S(\R^{2d-2})$ and every
$a,b,L\in\mathbb N_0$,
\[
\sup_{\substack{t\in I,\ f_1\in\mathcal B_1,\ f_2\in\mathcal B_2\\
                r,s\in\R}}
q\left(
C^{\rho_{0,t}}_{
    (1+|r|)^Lf_1^{(a)}(r),
    (1+|s|)^Lf_2^{(b)}(s)}
\right)
<\infty,
\]
where we used that $I$ is compact.
This shows that the
functions
\[
(v,r,s)\mapsto C^{\rho_{0,t}}_{f_1(r),f_2(s)}(v),
\qquad
t\in I,\quad f_1\in\mathcal B_1,\quad f_2\in\mathcal B_2,
\]
form a bounded subset of $\mathcal S(\R^{2d})$. Using again that $F \mapsto F \circ P$ is continuous on $\mathcal{S}(\R^{2d})$, this also shows the boundedness of the functions \eqref{eq:bounded_functions}.

Lastly, we show the claim \eqref{eq:amplitude_convergence}. For this, write
\begin{align*}
\big(A^t_{f_1,f_2}-A^{t_0}_{f_1,f_2} \big)(w,x,u)
&=e^{-2\pi it_0\alpha}
\left(
C^{\rho_{0,t}}_{f_1(u),f_2(x+u)}\bigl(\gamma(w,u)\bigr)
-
C^{\rho_{0,t_0}}_{f_1(u),f_2(x+u)}\bigl(\gamma(w,u)\bigr)
\right)\\
&\quad \quad \quad +
\left(e^{-2\pi it\alpha}-e^{-2\pi it_0\alpha}\right)
C^{\rho_{0,t}}_{f_1(u),f_2(x+u)}\bigl(\gamma(w,u)\bigr).
\end{align*}
The first term tends to zero in $\mathcal S(\R^{2d})$, uniformly for
$f_1\in\mathcal B_1$ and $f_2\in\mathcal B_2$, by
\eqref{eq:unmodulated_convergence} and the continuity of multiplication
by the fixed function $e^{-2\pi it_0\alpha}$, cf. \cite[Lemma 3.1]{howe1977on}. The second term tends to zero in $\mathcal{S}(\R^{2d})$ using \zcref{lem:howe} and the boundedness of the functions \eqref{eq:bounded_functions}. In combination, this shows the claim \eqref{eq:amplitude_convergence}.
\\~\\
\textbf{Step 3.} In this final step, we deduce the desired convergence of the matrix coefficients from the convergence of the amplitude functions established in Step 2. 

We start by rewriting the matrix coefficients. 
Let $\mathcal F_u$ denote the partial Fourier transform
in the $u$-variable, given by
\[
(\mathcal F_uF)(w,x,\xi)
=
\int_{\R}e^{-2\pi i\xi u}F(w,x,u)\,du,
\]
and, for $t>0$, define
\[
D_tF(y,w,x):=F(w,x,ty).
\]
Then the matrix coefficient formula \eqref{eq:coefficient-formula} can simply be written as
$
C^{\rho_t}_{f_1,f_2}
=
D_t\mathcal F_u A^t_{f_1,f_2}.
$
The partial Fourier transform is a continuous automorphism of
$\mathcal S(\R^{2d})$. Also, for each $t>0$, the map $D_t$ is a continuous automorphism of
$\mathcal S(\R^{2d})$, see, e.g., \cite[p. 346]{howe1977on}. Moreover, the map $t\mapsto D_t$ is continuous in the topology
of uniform convergence on bounded subsets. Since
$I \subseteq (0,\infty)$ is compact, the family $(D_t)_{t\in I}$ is
equicontinuous on $\mathcal{S}(\R^{2d})$.

For finishing the proof, we  write
\begin{align*}
C^{\rho_t}_{f_1,f_2}-C^{\rho_{t_0}}_{f_1,f_2}
&=
D_t\mathcal F_u
\left(A^t_{f_1,f_2}-A^{t_0}_{f_1,f_2}\right) +
(D_t-D_{t_0})\mathcal F_uA^{t_0}_{f_1,f_2}.
\end{align*}
The first term tends to zero in $\mathcal S(\R^{2d})$, uniformly for
$f_1\in\mathcal B_1$ and $f_2\in\mathcal B_2$, by
\eqref{eq:amplitude_convergence}, the continuity of $\mathcal F_u$,
and the equicontinuity of $(D_t)_{t\in I}$. Moreover, the functions
$\mathcal F_u A^{t_0}_{f_1,f_2}$, with $f_1\in\mathcal B_1$ and
$f_2\in\mathcal B_2$, form a bounded subset of
$\mathcal S(\R^{2d})$. Hence the second term tends to zero uniformly
because $D_t\to D_{t_0}$ uniformly on bounded subsets. In combination, this completes the
proof.
\end{proof}

\section{Coherent systems} \label{sec:coherent}
In this section, we consider coherent systems that are indexed by discrete subsets of the quotient group $G := N/Z(N)$ of a connected, simply connected, nonabelian, nilpotent Lie group $N$. Throughout, we let $\mu_G$ denote a fixed left Haar measure on $G$ and let $d$ denote a left-invariant Carnot-Carath\'eodory metric on $G$.

\subsection{Projective representations}
For $t> 0$, we let $(\rho_t, \Hrho)$ be the irreducible unitary representations of  $N$ that are square-integrable modulo $Z(N)$ as constructed in \zcref{sec:deformation_representation}.  In particular, we recall that we identify $\rho = \rho_1$. 

For a smooth cross-section $s \colon G \to N$ of the quotient map $p : N \to G$ as in \zcref{sec:malcev}, we define the maps
\[
\pi_t := \rho_t \circ s : G \to \mathcal{U}(\Hrho).
\]
Then each map $\pi_t$ is a projective unitary representation of $G$ on $\Hpi$, i.e., it satisfies $\pi_t (e) = I_{\Hpi}$ and
\[
    \pi_t(x)\pi_t(y) = \sigma_t(x,y)\pi_t(xy) , \quad x, y \in G,
\]
for the smooth $2$-cocycle $\sigma_t : G \times G \to \mathbb{T}$ given by
\[
    \sigma_t(x,y) = \chi_t(s(x)s(y)s(xy)^{-1}), \quad x,y \in G,
\]
where $\chi_t : Z(N) \to \mathbb{T}$ denotes the central character characterized by $ \rho_t|_{Z(N)} = \chi_t \cdot I_{\Hrho}$, see, e.g., \cite[Proposition 5]{aniello2006square}. Moreover, if $s' : G \to N$ is any other cross-section for $p : N \to G$ and $\pi'_t := \rho_t \circ s'$, then $\pi_t$ and $\pi'_t$ are related by $\pi_t(x) = v_t(x) \cdot \pi'_t (x)$ for a Borel function $v_t : G \to \mathbb{T}$, see, e.g., \cite[Proposition 3]{aniello2006square}. 
We use the common realization provided by \zcref{thm:continuity-in-t}, and identify $\pi = \pi_1$ using a fixed unitary intertwiner.

Recall that the central character of $\rho_t$ is given by $\chi_t(\exp_N(Z)) = e^{2\pi i t \ell(Z)}$ for $Z \in \mathfrak{z}(\mathfrak{n})$. Since $s(x)s(y)s(xy)^{-1} \in Z(N)$ for all $x,y \in G$, we may form
\[ \omega(x,y) = \ell \big( \exp_N^{-1}(s(x)s(y)s(xy)^{-1}) \big) , \quad x, y \in G, \]
so that the 2-cocycle can be written explicitly as
\[ \sigma_t(x,y) = e^{2\pi i t \omega(x,y)} , \quad x,y \in G. \]
We will use the following lemma on the convergence of 2-cocycles. Here, we recall that the weight functions $w_s : G \to [0, \infty)$, $s > 0$, are given as in \eqref{eq:weight}.

\begin{lemma}\label{lem:cocycle-convergence}
For any $s > 0$, we have that
\[ \lim_{t \to 1} \sup_{x,y \in G} \frac{ |\sigma_t(x,y) - \sigma(x,y) |}{ w_s(x) w_s(y) }  = 0 . \]
\end{lemma}

\begin{proof}
Fix $s > 0$. In addition to the crude estimate
\[
    |\sigma_t(x,y)-\sigma(x,y)| \leq 2,
\]
we have that
\[
    |\sigma_t(x,y)-\sigma(x,y)|
    = |e^{2\pi it\omega(x,y)}-e^{2\pi i\omega(x,y)}|
    \leq 2\pi |\omega(x,y)| |t-1|.
\]

Let $\varepsilon > 0$. We may find $R > 0$ such that
$w_s(x)>2/\varepsilon$ whenever $x\notin B_R(e) := \{ x \in G : d(x,e) < R \}$.
Since $w_s\geq1$, the two estimates above give
\begin{align*}
    \sup_{x,y\in G}
    \frac{|\sigma_t(x,y)-\sigma(x,y)|}{w_s(x)w_s(y)}
    &\leq
    \sup_{x,y\in B_R(e)}
    \frac{|\sigma_t(x,y)-\sigma(x,y)|}{w_s(x)w_s(y)}
    +
    \sup_{\substack{x,y\in G\\
        x\notin B_R(e)\ \text{or}\ y\notin B_R(e)}}
    \frac{|\sigma_t(x,y)-\sigma(x,y)|}{w_s(x)w_s(y)}
    \\
    &\leq
    2\pi |t-1|
    \sup_{x,y\in\overline{B_R(e)}}|\omega(x,y)|
    +\varepsilon.
\end{align*}
By compactness of $\overline{B_R(e)}$ and continuity of $\omega$,
the last supremum is finite. It follows that
\[
    \limsup_{t\to1}\sup_{x,y\in G}
    \frac{|\sigma_t(x,y)-\sigma(x,y)|}{w_s(x)w_s(y)}
    \leq\varepsilon.
\]
Since this holds for every $\varepsilon > 0$, the desired
conclusion follows.
\end{proof}

Lastly, we note that the projective representations $(\pi_t, \Hpi)$, $t>0$, of $G$ are irreducible and square-integrable with formal degree $d_{\pi_t} = d_{\rho_t} = t^{d} d_{\rho_1} = t^d d_{\pi}> 0$, see \zcref{lem:dim-scaled-rep}. The associated matrix coefficients on $G$ are denoted by $C^t_{f_1, f_2} : G \to \mathbb{C}$ and satisfy
 the covariance relation
\begin{align} \label{eq:covariance}
   \langle \pi_t (y) f_1, \pi_t (x) f_2 \rangle = \sigma_t (y, y^{-1} x) \langle f_1, \pi_t(y^{-1} x) f_2 \rangle, \quad x,y \in G.
\end{align}
Moreover, for a fixed vector $\phi \in \Hrho$ with $\| \phi \| = 1$, the map $V^t \colon \Hrho \to L^2(G)$ given by
\begin{equation}
    (V^t f)(x) = d_{\pi_t}^{1/2} \langle f, \pi_t(x) \phi \rangle, \quad f \in \Hrho, \; x \in G, \label{eq:coefficient-transform-isometry}
\end{equation}
is an isometry. Its adjoint is given weakly by
\begin{equation}
    (V^t)^*F = d_{\pi_t}^{1/2} \int_G F(x) \pi_t(x) \phi \; d\mu_G (x),   \label{eq:isometry-adjoint}
\end{equation}
for $F \in L^2(G)$.

\subsection{Relatively separated sets}
We denote the open ball by $B_r (x) := \{y \in G : d(x,y) < r \}$.
We recall that $G$ has polynomial growth, in the sense that there exists $c > 0$ and $d(G) \in \mathbb{N}$ such that
\begin{align} \label{eq:measure-ball-pol-growth}
\mu_G (B_r (x)) \leq c \cdot r^{d(G)}, \quad r \geq 1, \; x \in G,
\end{align}
 see, e.g., \cite{guivarch-croissance}.

A subset $\Lambda \subseteq G$ is called \emph{relatively separated} if for some (equivalently, all) $r > 0$ we have that
\[ \sup_{x \in G} \#( \Lambda \cap B_r(x) ) < \infty. \] 
From \eqref{eq:measure-ball-pol-growth}, it follows that there exists $c_{\Lambda} > 0$ such that
\begin{equation}
    \sup_{x \in G} \# ( \Lambda \cap B_r(x) ) \leq c_{\Lambda} \cdot r^{d(G)}, \quad r \geq 1, \; x \in G. \label{eq:rel-sep-pol-growth}
\end{equation}
 
Given $1 \leq p \leq \infty$, the Wiener amalgam space $W(L^p)$ consists of all measurable functions $F \colon G \to \C$ such that
\[ \| F \|_{W(L^p)}^p = \int_G \esssup_{y \in B_{1}(x)} |F(y)|^p \dif{\mu_G(x)} < \infty, \]
with the usual modification for $p = \infty$.
A continuous function $F : G \to \mathbb{C}$ is in $W(L^p)$ if and only if its restriction $F|_{\Lambda} \in \ell^p(\Lambda)$ for any relatively separated set $\Lambda \subseteq G$, cf. \cite[Lemma 3.8]{feichtinger-grochenig-banach}. In particular, we have the following estimate:
\begin{equation}
 \sup_{y \in G}   \sum_{\lambda \in \Lambda} |F(y^{-1} \lambda)| \leq C \| F \|_{W(L^1)}  \label{eq:sampling-inequality}
\end{equation}
for all continuous $F \in W(L^1)$, 
where $C > 0$ depends on $\Lambda$, cf. \cite[Lemma 2.4]{romero2021dual}.
Note that it follows from \eqref{eq:schwartz-intersection} that $\mathcal{S}(G) \subseteq W(L^p)$ for all $1 \leq p \leq \infty$.

We will repeatedly use the following observation:

\begin{lemma}\label{lem:counting_estimate}
Let $\Lambda \subseteq G$ be relatively separated. For all $s > d(G)+1$, we have that
\[ \int_G w_{- s} (x) \; d\mu_G (x) < \infty \quad \text{and} \quad \sup_{y \in G} \sum_{\lambda \in \Lambda} w_{-s}(y^{-1} \lambda) < \infty . \]
\end{lemma}

\begin{proof}
For the integrability claim, see, e.g., \cite[Proposition 1.5.1]{schweitzer1993dense}. 
The submultiplicativity of $w_{s}$ and integrability shows that $w_{-s} \in W(L^1)$ whenever $s > d(G) + 1$. Hence the result follows from the sampling inequality \eqref{eq:sampling-inequality}.
\end{proof}


\subsection{Frames and Riesz sequences}
For a smooth vector $\eta \in \Hpi^{\infty}$ and a relatively separated set $\Lambda \subseteq G$,
we consider the associated \emph{coherent system}
given by
\[ \pi_t (\Lambda) \eta = ( \pi_t (\lambda)\eta)_{\lambda \in \Lambda}. \]
Since $C^{\pi_t}_{\eta, \eta} \in \mathcal{S}(G)$ by \zcref{thm:schwartz_matrix} and $\mathcal{S}(G) \subseteq W(L^1)$, 
it follows from the sampling inequality \eqref{eq:sampling-inequality} that the smooth coherent system $\pi_t (\Lambda) \eta$ is a \emph{Bessel sequence}, i.e., the associated analysis operator $\anal^t : \Hpi \to \ell^2 (\Lambda)$ and synthesis operator $\synt^t : \ell^2 (\Lambda) \to \Hpi$ given by
\[
\anal^t f = \big( \langle f, \pi_t (\lambda) \eta \rangle \big)_{\lambda \in \Lambda} \quad \text{and} \quad \synt^t c = \sum_{\lambda \in \Lambda} c_\lambda \pi_t(\lambda) \eta,
\]
respectively, are well-defined and bounded. Note that $\synt^t = (\anal^t)^{\ast}$. The positive operators $\fram^t := \synt^t \anal^t$ and $\gram^t := \anal^t \synt^t$ are called the \emph{frame operator} and \emph{Gramian} associated to the coherent system $\pi_t(\Lambda) \eta$, respectively. We simply write $\anal, \synt, \fram, \gram$ when $t=1$.

The coherent system $\pi_t(\Lambda)\eta$ is said to be a \emph{frame} whenever $\synt^t$ has a bounded right inverse; equivalently, if its frame operator $\fram^t$ is invertible. 
We call $\pi_t(\Lambda)\eta$ a \emph{frame sequence} if $\fram^t$ is invertible on the span $M := \overline{\mathrm{span}} \{ \pi_t(\lambda) \eta : \lambda \in \Lambda \}$. Lastly, $\pi_t(\Lambda)\eta$ is called a \emph{Riesz sequence} if  $\synt^t$ has a bounded left inverse; equivalently, if $\gram^t$ is invertible on $\ell^2 (\Lambda)$. A Riesz sequence is in particular a frame sequence.

When $\pi(\Lambda)\eta$ is a frame sequence, we will frequently consider its \emph{canonical dual frame sequence} $(\tilde{\eta}_\lambda)_{\lambda \in \Lambda}$ given by $\tilde{\eta}_\lambda = (\fram|_M)^{-1} \pi(\lambda) \eta$ for $\lambda \in \Lambda$. Let $\gram^\dagger$ denote the Moore--Penrose inverse of $\gram$. Then the identity $(\fram|_M)^{-1} \synt = \synt \gram^\dagger$ gives the following description of the canonical dual frame sequence:
\begin{equation}
    \tilde{\eta}_\lambda = \sum_{\lambda' \in \Lambda} \gram_{\lambda',\lambda}^\dagger \pi(\lambda')\eta, \quad \lambda \in \Lambda. \label{eq:dual-frame-pseudo}
\end{equation}
When $\pi(\Lambda)\eta$ is a frame for $\mathcal{H}$, then $(\tilde{\eta}_\lambda)_{\lambda \in \Lambda}$ is called the \emph{canonical dual frame} of $\pi(\Lambda)\eta$, and it satisfies
\begin{equation}
    f = \sum_{\lambda \in \Lambda} \langle f, \pi(\lambda) \eta \rangle \tilde{\eta}_{\lambda} = \sum_{\lambda \in \Lambda} \langle f, \tilde{\eta}_\lambda \rangle \pi(\lambda) \eta, \quad f \in \mathcal{H}. \label{eq:dual-frame}
\end{equation}
 If $\pi(\Lambda)\eta$ is a Riesz sequence, its canonical dual frame sequence $(\tilde{\eta}_\lambda)_{\lambda \in \Lambda}$ is \emph{biorthogonal} to $\pi(\Lambda)\eta$ in the sense that
\begin{equation}
    \langle \pi(\lambda)\eta, \tilde{\eta}_{\lambda'} \rangle = \delta_{\lambda,\lambda'} , \quad \lambda, \lambda' \in \Lambda. \label{eq:biorthognal}
\end{equation}
We refer to the book \cite{christensen2016introduction} for standard background on frames and Riesz sequences. 

We will use the following lemma on bounded sequences in $\Hpi^{\infty}$ generating Bessel sequences.

\begin{lemma}\label{lem:bessel-sequence-t}
Let $(\phi_\lambda)_{\lambda \in \Lambda}$ be a bounded sequence in $\mathcal{H}^\infty$. Then for any $t > 0$ the sequence $( \pi_t(\lambda) \phi_\lambda )_{\lambda \in \Lambda}$ is a Bessel sequence.
\end{lemma}

\begin{proof}
Fix $t>0$. Since $(\phi_\lambda)_{\lambda \in \Lambda}$ is assumed bounded in $\mathcal{H}^\infty$, \zcref{thm:schwartz_matrix} implies that for every $s>0$,
there exists $K_{s}>0$ such that
\[
    |C^t_{\phi_{\lambda},\phi_{\lambda'}}(x)|
    \leq K_{s}w_{-s}(x),
    \qquad x\in G,\quad \lambda,\lambda' \in\Lambda.
\]
Consider the matrix
$
    A^t_{\lambda',\lambda}
    :=\langle \pi_t(\lambda)\phi_{\lambda},
              \pi_t(\lambda')\phi_{\lambda'}\rangle
$ for $\lambda, \lambda' \in \Lambda$.
By the covariance relation \eqref{eq:covariance}, we have
\[
    |A^t_{\lambda',\lambda}|
    =|C^t_{\phi_{\lambda},\phi_{\lambda'}}(\lambda^{-1}\lambda')|
    \leq K_{s}w_{-s}(\lambda^{-1}\lambda').
\]
Upon choosing $s>d(G)+1$, it follows from \zcref{lem:counting_estimate} that
\[
    \sup_{\lambda\in\Lambda}
    \sum_{\lambda'\in\Lambda}|A^t_{\lambda',\lambda}|<\infty, \quad \sup_{\lambda' \in\Lambda}
    \sum_{\lambda\in\Lambda}|A^t_{\lambda',\lambda}|<\infty.
\]
Therefore, 
Schur's test shows that $A^t$ defines a bounded
operator on $\ell^2(\Lambda)$, which implies that
that $(\pi_t(\lambda)\phi_{\lambda})_{\lambda \in\Lambda}$
is a Bessel sequence.
\end{proof}

In particular, we will apply \zcref{lem:bessel-sequence-t} to the following bounded sequences.

\begin{lemma}\label{lem:boundedness-frame-sequence}
Let $\eta \in \Hrho^\infty$ and suppose that $\pi(\Lambda) \eta$ is a frame sequence with canonical dual frame sequence $(\tilde{\eta}_\lambda)_{\lambda \in \Lambda}$. Then the system
$ \big( \pi(\lambda)^* \tilde{\eta}_{\lambda} \big)_{\lambda \in \Lambda}$
is bounded in $\Hpi^{\infty}$.
\end{lemma}

\begin{proof}
The proof will use the following matrix algebra:
\[ \mathcal{J}(\Lambda) = \Big\{ A \in \mathcal{B}(\ell^2(\Lambda)) \; : \; \forall s > 0 \; \sup_{\lambda,\lambda' \in \Lambda}|A_{\lambda,\lambda'}| w_s(\lambda^{-1}\lambda') < \infty \Big\} . \]
This is a special case of the so-called \emph{Jaffard matrix algebra} \cite{jaffard-1990} studied in \cite{sun2007}; see \zcref{sec:spectral-invariance} for details.

Let $\gram \in \mathcal{B}(\ell^2(\Lambda))$ denote the Gramian of $\pi(\Lambda) \eta$. Denote by $\gram^{\dagger}$ its Moore-Penrose inverse. Then by \eqref{eq:dual-frame-pseudo} the canonical dual frame sequence of $\pi(\Lambda)\eta$ can be expressed as
\[ \tilde{\eta}_\lambda = \sum_{\lambda' \in \Lambda} \gram_{\lambda',\lambda}^\dagger \pi(\lambda')\eta , \quad \lambda \in \Lambda . \]
We will use this expression and an estimate for the action of $\pi$ on the seminorms defining $\Hpi^{\infty}$ to prove the claim.

We start by deriving pointwise estimates on the entries defining $\gram^{\dagger}$. For this, note that, using the covariance relation \eqref{eq:covariance}, the entries of $\gram$ are given by
\[ \gram_{\lambda',\lambda} = \langle \pi(\lambda) \eta ,\pi(\lambda')\eta \rangle = \sigma(\lambda, \lambda^{-1} \lambda') C_{\eta,\eta}(\lambda^{-1}\lambda'), \quad \lambda,\lambda' \in \Lambda. \]
By \zcref{thm:schwartz_matrix}, the matrix coefficient $C_{\eta,\eta}$ is a Schwartz function on $G$, so by the description of $\mathcal{S}(G)$ in \eqref{eq:schwartz-intersection}, it follows easily that $\gram \in \mathcal{J}(\Lambda)$. By \zcref{thm:spectral-invariance}, it follows that its Moore--Penrose inverse $\gram^\dagger$ also belongs to $\mathcal{J}(\Lambda)$. Hence, letting $s > 0$ be arbitrary, there exists $C' > 0$ such that
\begin{equation}
    |\gram_{\lambda',\lambda}^\dagger| \leq C'  w_{-s}(\lambda^{-1}\lambda'), \quad \lambda,\lambda' \in \Lambda. \label{eq:gramian_bound}
\end{equation}

We will next show an estimate for the action of $\pi$ on $\Hpi^{\infty}$. For this, recall that $\pi(x) = \rho(s(x))$ for a unitary representation $(\rho, \Hpi)$ of $N$. Fix $\alpha \in \N_0^n$. A direct calculation gives, for $n \in N$,
\[
    d\rho(X^\alpha)\rho(n)\eta
    =
    \rho(n)\sum_{|\beta|\leq|\alpha|}
    c_{\alpha,\beta}(n)d\rho(X^\beta)\eta
\]
for functions $c_{\alpha, \beta} : N \to \C$ arising from the adjoint action, see, e,g., the proof of \cite[Lemma A.1.1]{corwin1990representations}.
Since $N$ is nilpotent, the functions $c_{\alpha, \beta}$ are polynomials on $N$ that may be chosen to be constant on cosets $nZ(N)$, so that
$c_{\alpha,\beta}\circ s$ are polynomial functions on $G=N/Z(N)$.
By \cite[Section 1.5]{ludwig1988minimal} or \cite[Section 3.5]{ludwig1995schwartz}, these polynomials can be estimated by weights, in the sense that there exist
$C_\alpha,r_\alpha>0$ such that
$
    |c_{\alpha,\beta}(s(x))|
    \leq C_\alpha w_{r_\alpha}(x)
$
for $ x\in G$ and $|\beta|\leq|\alpha|$.
Therefore, we obtain that 
\begin{align} \label{eq:seminorm-continuity}
    p_\alpha(\pi(x)\eta)
    \leq C_\alpha w_{r_\alpha}(x)
    \sum_{|\beta|\leq|\alpha|}p_\beta(\eta)
    .
\end{align}
For simplicity, we set $p'(\eta) := \sum_{|\beta| \leq |\alpha|} p_{\beta} (\eta)$.

Lastly, using that
\[ \pi(\lambda)^* \tilde{\eta}_\lambda = \sum_{\lambda' \in \Lambda} \gram_{\lambda',\lambda}^\dagger \overline{\sigma(\lambda, \lambda^{-1} \lambda')} \pi(\lambda^{-1}\lambda') \eta , \]
in combination with \eqref{eq:seminorm-continuity}, we obtain
\[ p_{\alpha} (\pi(\lambda)^* \tilde{\eta}_\lambda) \leq \sum_{\lambda' \in \Lambda} |\gram_{\lambda,\lambda'}^\dagger| C_{\alpha} w_r(\lambda^{-1}\lambda') p'(\eta) \leq C_{\alpha} C' p'(\eta) \sum_{\lambda' \in \Lambda} w_{r-s}(\lambda^{-1}\lambda')  .  \]
Choosing $s > 0$ such that $s-r > d(G)+1$, we thus get using \zcref{lem:counting_estimate} that
\[
\sup_{\lambda \in \Lambda} p_{\alpha} (\pi(\lambda)^* \tilde{\eta}_{\lambda} ) < \infty, \quad \alpha \in \N_0^n.
\]
This proves that $(\pi(\lambda)^*\tilde\eta_\lambda)_{\lambda \in \Lambda}$ is bounded in $\Hpi^{\infty}$.
\end{proof}

\subsection{Deformation of frames and Riesz sequences}

The following theorem is the main technical deformation result on coherent systems that we need in order to establish \zcref{thm:main_intro}. It builds fundamentally on \zcref{thm:continuity-in-t} as well as \zcref{lem:boundedness-frame-sequence}.

\begin{theorem}\label{thm:common-deformation}
Suppose that $\pi(\Lambda) \eta$ is a frame sequence for $\Hpi$ with $\eta \in \Hpi^{\infty}$ and denote by $( \tilde{\eta}_{\lambda} )_{\lambda \in \Lambda}  := \fram^{-1} \pi(\Lambda) \eta $ its canonical dual frame sequence. 
Let $\synt^t$ denote the synthesis operator of  $\pi_t(\Lambda)\eta$ and let $\tilde{\anal}^t$ denote the analysis operator of 
\begin{equation}
    ( \pi_t(\lambda) \pi(\lambda)^* \tilde{\eta}_\lambda )_{\lambda \in \Lambda}. \label{eq:coherent-system-t}
\end{equation}
The following assertions hold:
\begin{enumerate}
    \item[(i)] If $\pi(\Lambda)\eta$ is a frame, then $\tilde{\fram}^t \coloneqq \synt^t \tilde{\anal}^t \to I$ in operator norm as $t \to 1$.
    \item[(ii)] If $\pi(\Lambda)\eta$ is a Riesz sequence, then $\tilde{\gram}^t \coloneqq \tilde{\anal}^t \synt^t \to I$ in operator norm as $t \to 1$.
\end{enumerate}
\end{theorem}

\begin{proof}
Abbreviate $\phi_\lambda := \pi(\lambda)^* \tilde{\eta}_\lambda$ for $\lambda \in \Lambda$. We may without loss of generality assume that $\| \eta \| = 1$. An application of \zcref{lem:boundedness-frame-sequence} yields that the set $\mathcal{B} := \{ \eta \} \cup \{ \phi_{\lambda} : \lambda \in \Lambda \}$ is bounded in $\Hpi^{\infty}$. Then \zcref{lem:bessel-sequence-t} shows that the system in \eqref{eq:coherent-system-t} is Bessel, so all operators considered are bounded.

We may also apply \zcref{thm:continuity-in-t} and conclude that $C_{f_1,f_2}^t \to C_{f_1,f_2}^1$ in $\mathcal{S}(G)$ as $t \to 1$ uniformly for all $f_1,f_2 \in \mathcal{B}$. Let $I \subseteq (0, \infty)$ be a compact interval containing $1$ in its interior. For every $s > 0$, this gives a function $\varepsilon_s(t) \to 0$ and a constant $C_s > 0$ such that, for all $f_1,f_2 \in \mathcal{B}$ and all $t \in I$,
\begin{align}
    | C_{f_1,f_2}^t(x) - C_{f_1,f_2}^1(x) | &\leq \varepsilon_s(t) w_{-s}(x) , \label{eq:schwartz-inequality} \\
    | C_{f_1,f_2}^t (x) | &\leq C_s \cdot w_{-s}(x) . \label{eq:schwartz-bounded}
\end{align}
By \zcref{lem:counting_estimate}, we have that
\begin{align}
    I_s &:= \int_G w_{-s}(x) \; d\mu_G (x) < \infty, \label{eq:integral-finite} \\
    J_s &:= \sup_{y \in G} \sum_{\lambda \in \Lambda} w_{-s}(y^{-1}\lambda) < \infty\label{eq:sum-finite}
\end{align}
whenever $s > d(G) + 1$. We  prove assertions (i) and (ii) separately.
\\~\\
(i) Suppose $\pi(\Lambda)\eta$ is a frame for $\Hrho$. Consider the isometry $V^t \colon \Hrho \to L^2(G)$ from \eqref{eq:coefficient-transform-isometry} given by $(V^t f)(x) = d_{\pi_t}^{1/2} \langle f, \pi_t(x) \eta \rangle$ for $f \in \Hrho$ and $x \in G$. Conjugating $\tilde{\fram}^t-I$ by $V^t$, we obtain
\begin{align*}
    \big(  V^t (\tilde{\fram}^t-I) (V^t)^* \big) F \big)(x) &= d_{\pi_t}^{1/2} \big\langle (\tilde{\fram}^t - I) (V^t)^* F, \pi_t(x) \eta \big\rangle \\
    &= d_{\pi_t} \int_G F(y) \big\langle  (\tilde{\fram}^t - I) \pi_t(y) \eta, \pi_t(x) \eta \big\rangle \; d\mu_G (y) .
\end{align*}
This is an integral operator on $L^2(G)$ with kernel equal to
\begin{align*}
    K_t(x,y) &= d_{\pi_t} \langle (\tilde{\fram}^t - I) \pi_t(y) \eta, \pi_t(x) \eta \rangle \\
    &= d_{\pi_t} \sum_{\lambda \in \Lambda} \langle \pi_t(y) \eta, \pi_t(\lambda) \phi_{\lambda} \rangle \langle \pi_t(\lambda)\eta, \pi_t(x) \eta \rangle - d_{\pi_t} \langle \pi_t(y)\eta, \pi_t(x)\eta \rangle \\
    &= d_{\pi_t} \sigma_t(y,y^{-1}x) \Big( \sum_{\lambda \in \Lambda} \sigma_t(y^{-1}\lambda,\lambda^{-1}x) C_{\eta,\phi_{\lambda}}^t(y^{-1}\lambda) C_{\eta, \eta}^t(\lambda^{-1}x) - C_{\eta,\eta}^t(y^{-1}x) \Big) .
\end{align*}
To ease notation, we set
\[ \tilde{K}_t(x,y) := d_{\pi_t}^{-1} \overline{ \sigma_t(y,y^{-1}x)} K_t(x,y), \quad x,y \in G. \]
Then $| K_t(x,y) | = d_{\pi_t} | \tilde{K}_t(x,y) |$ for all $x,y \in G$.

Note that for $t = 1$, the operator $\tilde{\fram}^1$ is the mixed frame operator for $\pi(\Lambda)\eta$ and its canonical dual frame, so that $\tilde{\fram}^1 = I_{\Hpi}$ by \eqref{eq:dual-frame}. This gives that $K_1(x,y) = 0$ for all $x,y \in G$, and thus $\tilde{K}_1(x,y) = 0$ for all $x, y \in G$. This leads us to consider
\begin{align*}
    &\tilde{K}_t(x,y) \\
    &= \tilde{K}_t(x,y) - \tilde{K}_1(x,y) \\
    &= \sum_{\lambda \in \Lambda} \big( \sigma_t(y^{-1}\lambda,\lambda^{-1}x) C_{\eta,\phi_{\lambda}}^t(y^{-1}\lambda) C_{\eta, \eta}^t(\lambda^{-1}x) - \sigma_1(y^{-1}\lambda,\lambda^{-1}x) C_{\eta,\phi_{\lambda}}^1(y^{-1}\lambda) C_{\eta, \eta}^1(\lambda^{-1}x) \big) \\
& \quad \quad + C_{\eta,\eta}^1(y^{-1}x) - C_{\eta,\eta}^t(y^{-1}x) . \numberthis \label{eq:final-term}
\end{align*}
Using telescoping, we may write the summand above as a sum of three terms:
\begin{align}
    & \bigg( \sigma_t(y^{-1}\lambda,\lambda^{-1}x) - \sigma_1(y^{-1}\lambda,\lambda^{-1}x) \bigg) C_{\eta,\phi_\lambda}^t(y^{-1}\lambda) C_{\eta, \eta}^t(\lambda^{-1}x), \label{eq:frame-proof-summand-1} \\
    & \sigma_1(y^{-1}\lambda,\lambda^{-1}x) \bigg( C_{\eta,\phi_\lambda}^t(y^{-1}\lambda) - C_{\eta,\phi_\lambda}^1(y^{-1}\lambda) \bigg) C_{\eta, \eta}^t(\lambda^{-1}x), \label{eq:frame-proof-summand-2} \\
    & \sigma_1(y^{-1}\lambda,\lambda^{-1}x) C_{\eta,\phi_\lambda}^1(y^{-1}\lambda) \bigg( C_{\eta, \eta}^t(\lambda^{-1}x) - C_{\eta, \eta}^1(\lambda^{-1}x) \bigg) . \label{eq:frame-proof-summand-3}
\end{align}
We will next estimate each of these three terms.
Consider first the term in \eqref{eq:frame-proof-summand-1}. Fix $s' > 0$. Using \zcref{lem:cocycle-convergence}, there exists $\varepsilon'_{s'}(t) \to 0$ as $t \to 1$ such that
\[ \big| \sigma_t(y^{-1}\lambda,\lambda^{-1}x) - \sigma_1(y^{-1}\lambda,\lambda^{-1}x) \big| \leq \varepsilon'_{s'}(t) w_{s'}(y^{-1}\lambda) w_{s'}(\lambda^{-1}x) . \]
Fixing another $s > 0$ and applying \eqref{eq:schwartz-bounded}, we deduce that
\begin{align}
    &\Bigg| \bigg( \sigma_t(y^{-1}\lambda,\lambda^{-1}x) - \sigma_1(y^{-1}\lambda,\lambda^{-1}x) \bigg) C_{\eta,\phi_\lambda}^t(y^{-1}\lambda) C_{\eta, \eta}^t(\lambda^{-1}x) \Bigg| \nonumber \\
    &\quad \quad \quad \quad \leq \varepsilon'_{s'}(t) w_{s'}(y^{-1}\lambda) w_{s'}(\lambda^{-1}x) C_{s} w_{-s}(y^{-1}\lambda) C_s w_{-s}(\lambda^{-1}x) \nonumber \\
    & \quad \quad \quad \quad \leq C_s^2 \varepsilon'_{s'}(t) w_{s'-s}(y^{-1}\lambda) w_{s'-s}(\lambda^{-1}x) . \label{eq:estimate-1}
\end{align}
For \eqref{eq:frame-proof-summand-2}, we apply \eqref{eq:schwartz-inequality} and \eqref{eq:schwartz-bounded} to obtain
\begin{align}
    &\Big|  \sigma_1(y^{-1}\lambda,\lambda^{-1}x) \big( C_{\eta,\phi_\lambda}^t(y^{-1}\lambda) - C_{\eta,\phi_\lambda}^1(y^{-1}\lambda) \big) C_{\eta, \eta}^t(\lambda^{-1}x) \Big| \nonumber \\
    &\quad \quad \quad \quad  \leq \varepsilon_s(t) C_s w_{-s}(y^{-1}\lambda) w_{-s}(\lambda^{-1}x). \label{eq:estimate-2}
\end{align}
For \eqref{eq:frame-proof-summand-3}, we also apply \eqref{eq:schwartz-inequality} and \eqref{eq:schwartz-bounded}, which yields
\begin{align}
    &\Big| \sigma_1(y^{-1}\lambda,\lambda^{-1}x) C_{\eta,\phi_\lambda}^1(y^{-1}\lambda) \big( C_{\eta, \eta}^t(\lambda^{-1}x) - C_{\eta, \eta}^1(\lambda^{-1}x) \big) \Big| \nonumber \\
    &\leq \varepsilon_s(t) C_s w_{-s}(y^{-1}\lambda) w_{-s}(\lambda^{-1}x) . \label{eq:estimate-3}
\end{align}
In addition, we estimate the final term \eqref{eq:final-term} also using \eqref{eq:schwartz-inequality} as follows:
\begin{equation}
    |C_{\eta,\eta}^1(y^{-1}x) - C_{\eta,\eta}^t(y^{-1}x)| \leq \varepsilon_s(t) w_{-s}(y^{-1}x). \label{eq:estimate-4}
\end{equation}
Combining \eqref{eq:estimate-1}, \eqref{eq:estimate-2}, \eqref{eq:estimate-3} and \eqref{eq:estimate-4}, we arrive at
\begin{align*}
    |\tilde{K}_t(x,y)| &\leq C_s^2 \varepsilon'_{s'}(t) \sum_{\lambda \in \Lambda} w_{s'-s}(y^{-1}\lambda) w_{s'-s}(\lambda^{-1}x) + \varepsilon_s(t) C_s \sum_{\lambda \in \Lambda} w_{-s}(y^{-1}\lambda)  w_{-s}(\lambda^{-1}x) \\
    &\quad \quad + \varepsilon_s(t) C_s \sum_{\lambda \in \Lambda} w_{-s}(y^{-1}\lambda) w_{-s}(\lambda^{-1}x) + \varepsilon_s(t)   w_{-s}(y^{-1}x) \\
    &\leq (C_s^2 \varepsilon'_{s'}(t) + 2 C_s \varepsilon_s(t)) \sum_{\lambda \in \Lambda} w_{s'-s}(y^{-1}\lambda) w_{s'-s}(\lambda^{-1}x) + \varepsilon_s(t) w_{-s}(y^{-1}x) .
\end{align*}
Choose $s,s' > 0$ such that $s-s' > d(G) + 1$. Recall that $d_{\pi_t} = t^d d_{\pi}$ for some $d \in \mathbb{N}$.
Integrating over $ G$ and applying \eqref{eq:integral-finite} and \eqref{eq:sum-finite}, we obtain, for any $y \in G$,
\begin{align*}
    \int_G |K_t(x,y)| \; d\mu_G (x)
    & \leq d_{\pi_t} \int_G |\tilde{K}_t(x,y)| \; d\mu_G (x) \\
    & \leq  d_{\pi_t} (C_s^2 \varepsilon'_{s'} (t) + 2 C_s \varepsilon_s(t)) \sum_{\lambda \in \Lambda} w_{s'-s}(y^{-1}\lambda) \int_G w_{s'-s}(\lambda^{-1}x) \; d\mu_G (x) \\
    &\quad \quad +  d_{\pi_t} \varepsilon_s(t) \int_G w_{-s}(y^{-1}x) \; d\mu_G (x) \\
    &  \leq t^{d} d_\pi (C_s^2 \varepsilon'_{s'}(t) + 2 C_s \varepsilon_s(t)) J_{s-s'} I_{s-s'} + t^{d} d_\pi \varepsilon_s(t) I_s  ,
\end{align*}
which implies that
\[
\esssup_{y \in G} \int_G |K_t (x,y) | \; d\mu_G (x) \to 0 \quad \text{as} \quad t \to 1.
\]
By symmetry of the bound for $|K_t(x,y)|$, the same argument shows that $\int_G | K_t(x,y) | d\mu_G (y) \to 0$ uniformly in $x \in G$ as $t \to 1$. By Schur's test, we therefore conclude that
\[ \big \| V^t (\tilde{\fram}^t - I) (V^t)^* \big \|_{\mathcal{B}(L^2(G))} \to 0 \quad \text{as } t \to 1 . \]
Since $V^t$ is an isometry, this implies that $\tilde{\fram}^t \to I$ in operator norm on $\mathcal{B}(\mathcal{H})$ as $t \to 1$, which finishes the proof of (i).
\\~\\
(ii)  Suppose $\pi(\Lambda)\eta$ is a Riesz sequence in $\Hrho$. The matrix representation of $\tilde{\gram}^t = \tilde{\anal}^t \synt^t$ is given by
\begin{align}
    (\tilde{\gram}^t)_{\lambda',\lambda} &= \langle \pi_t(\lambda) \eta, \pi_t(\lambda') \phi_{\lambda'} \rangle \nonumber \\
    &= \sigma_t(\lambda,\lambda^{-1}\lambda') \langle \eta, \pi_t(\lambda^{-1}\lambda') \phi_{\lambda'} \rangle \nonumber \\
    &= \sigma_t(\lambda,\lambda^{-1}\lambda') C_{\eta,\phi_{\lambda'}}^t(\lambda^{-1}\lambda') , \label{eq:gramian-expression}
\end{align}
where we used the covariance relation \eqref{eq:covariance}.
Since $(\tilde{\eta}_\lambda)_{\lambda \in \Lambda}$ is biorthogonal to $\pi(\Lambda)\eta$, see \eqref{eq:biorthognal}, we obtain that $\tilde{\gram}^1 = I$, so \eqref{eq:gramian-expression} gives
\begin{equation}
    \big| C_{\eta,\phi_{\lambda'}}^1(\lambda^{-1}\lambda') \big| = 0, \quad \lambda,\lambda' \in \Lambda,\; \lambda \neq \lambda'. \label{eq:orthogonality-in-terms-of-c}
\end{equation}

We  consider two cases.
If $\lambda = \lambda'$, then the biorthogonality of $\pi(\Lambda)\eta$ and $(\tilde{\eta}_\lambda)_{\lambda \in \Lambda}$ gives
\[
    \big| (\tilde{\gram}^t - I)_{\lambda,\lambda} \big|
    = \big|(\tilde{\gram}^t)_{\lambda,\lambda} - 1\big|
    = \big|\langle \pi_t(\lambda)\eta, \pi_t(\lambda)\pi(\lambda)^*\tilde{\eta}_\lambda \rangle - 1\big|
    = \big|\langle \pi(\lambda)\eta, \tilde{\eta}_\lambda \rangle - 1\big|
    = 0.
\]
On the other hand, if $\lambda \neq \lambda'$, then \eqref{eq:gramian-expression} and \eqref{eq:orthogonality-in-terms-of-c} together with \eqref{eq:schwartz-inequality} give
\begin{align*}
    \big| (\tilde{\gram}^t - I)_{\lambda',\lambda} \big|
    &= \big| \tilde{\gram}^t_{\lambda',\lambda} \big|
    = \big|C_{\eta,\phi_{\lambda'}}^t(\lambda^{-1}\lambda')\big| \\
    &= \big|C_{\eta,\phi_{\lambda'}}^t(\lambda^{-1}\lambda')
          - C_{\eta,\phi_{\lambda'}}^1(\lambda^{-1}\lambda')\big| \\
    &\leq \varepsilon_s(t) w_{-s}(\lambda^{-1}\lambda').
\end{align*}

Thus, for all $\lambda,\lambda' \in \Lambda$, we have the estimate
\begin{equation*}
    \big| (\tilde{\gram}^t - I)_{\lambda',\lambda} \big|
    \leq \varepsilon_s(t) w_{-s}(\lambda^{-1}\lambda'),
    \quad t \in I.
\end{equation*}
Letting $s > d(G)+1$, \eqref{eq:sum-finite} yields
\[
    \sup_{\lambda \in \Lambda}
    \sum_{\lambda' \in \Lambda}
    |(\tilde{\gram}^t - I)_{\lambda',\lambda}|
    \leq \varepsilon_s(t) \cdot J_s
    \to 0 \quad \text{as } t \to 1.
\]
 By Schur's test, we conclude that
\[
    \big\| \tilde{\gram}^t - I \big\|_{\mathcal{B}(\ell^2(\Lambda))}
    \to 0 \quad \text{as } t \to 1.
\]
This finishes the proof.
\end{proof}

We obtain the following important consequence of \zcref{thm:common-deformation}. 

\begin{corollary}\label{cor:frame-riesz-stability}
Let $\eta \in \Hrho^\infty$. Then the following assertions hold:
\begin{enumerate}
    \item[(i)] If $\pi(\Lambda)\eta$ is a frame for $\Hrho$, then $\pi_t(\Lambda)\eta$ is a frame for $\Hrho$ for all $t > 0$ sufficiently close to $1$.
    \item[(ii)] If $\pi(\Lambda)\eta$ is a Riesz sequence in $\Hrho$, then $\pi_t(\Lambda)\eta$ is a Riesz sequence in $\Hrho$ for all $t > 0$ sufficiently close to $1$.
\end{enumerate}
\end{corollary}

\begin{proof}
Let $\synt^t$ denote the synthesis operator of $\pi_t(\Lambda) \eta$ and let $\tilde{\anal}^t$ denote the analysis operator of $(\pi_t(\lambda) \pi(\lambda)^* \tilde{\eta}_\lambda )_{\lambda \in \Lambda}$, where $(\tilde{\eta}_{\lambda})_{\lambda \in \Lambda}$ is the canonical dual frame sequence of $\pi(\Lambda) \eta$.
\\~\\
(i) Suppose $\pi(\Lambda)\eta$ is a frame for $\Hrho$. Then, by \zcref{thm:common-deformation}, we have that $\tilde{\fram}^t := \synt^t \tilde{\anal}^t \to I_{\Hpi}$ in operator norm as $t \to 1$. In particular, $\tilde{\fram}^t$ is invertible for $t$ sufficiently close to $1$, and
\[ \synt^t (\tilde{\anal}^t (\tilde{\fram}^t)^{-1}) = I_{\Hpi} . \]
This shows that $\synt^t$ has a bounded right inverse for $t$ sufficiently close to $1$, showing that $\pi_t(\Lambda)\eta$ is a frame for such $t$.
\\~\\
(ii) Supposing $\pi(\Lambda)\eta$ is a Riesz sequence in $\Hrho$. Then \zcref{thm:common-deformation}  gives that $\tilde{\gram}^t = \tilde{\anal}^t \synt^t \to I$ in operator norm as $t \to 1$. Hence, $\tilde{\gram}^t$ is invertible for $t$ sufficiently close to $1$, and
\[ ((\tilde{\gram}^t)^{-1} \tilde{\anal}^t) \synt^t = I_{\ell^2 (\Lambda)} . \]
This shows that $\synt^t$ has a bounded left inverse for all $t$ sufficiently close to $1$, which implies that $\pi_t(\Lambda)\eta$ is a Riesz sequence in $\Hrho$ for such $t$.
\end{proof}

\subsection{Density conditions for coherent systems} \label{sec:density}
Let $(K_n)_{n \in \mathbb{N}}$ be any strong F\o lner sequence in $G$, i.e., a sequence of nonnull compact sets $K_n \subseteq G$ satisfying
\[
\lim_{n \to \infty} \frac{\mu_G (K_n K \cap K_n^c K)}{\mu_G (K_n)} = 0
\]
for all compact $K \subseteq G$. For example, if $(r_n)_{n \in \mathbb{N}}$ is a sequence in $(0, \infty)$ such that $r_n \to \infty$, then the closed balls $K_n := \overline{B_{r_n} (e)}$ associated to the Carnot-Carath\'edory metric form a strong F\o lner sequence, see \cite[Corollary 1.10]{breuillard2014geometry}.

For a relatively separated set $\Lambda \subseteq G$, we define its lower and upper Beurling density by
\[
D^- (\Lambda) := \lim_{n \to \infty} \inf_{x \in G} \frac{\# (\Lambda \cap x K_n)}{\mu_G (K_n)} 
\]
and
\[
D^+ (\Lambda) := \lim_{n \to \infty} \sup_{x \in G} \frac{\# (\Lambda \cap x K_n)}{\mu_G (K_n)},  
\]
respectively. The Beurling densities are independent of the choice of strong F\o lner sequence, cf. \cite[Proposition 5.14]{pogorzelski2022leptin}.

Using the results from the previous subsection, we  prove \zcref{thm:main_intro}.

\begin{proof}[Proof of \zcref{thm:main_intro}]
Recall that, by \zcref{lem:dim-scaled-rep}, we have $d_{\pi_t} = t^d d_{\pi}$ for some $d \in \mathbb{N}$. 
\\~\\
(i) Suppose that $\pi(\Lambda) \eta$ is a frame for $\Hpi$ with $\eta \in \Hrho^\infty$. By \zcref{cor:frame-riesz-stability}, there exists $t > 1$ such that $\pi_t(\Lambda) \eta$ is a frame for $\Hrho$. The necessary density conditions for frames (see \zcref{prop:necessary_density_conditions}) then give that $D^-(\Lambda) \geq d_{\pi_t}$, which implies that
\[ D^-(\Lambda) \geq d_{\pi_t} = t^d d_\pi > d_\pi , \]
which proves the claim.
\\~\\
(ii) Suppose that $\pi(\Lambda) \eta$ is a Riesz sequence in $\Hrho$ with $\eta \in \Hrho^\infty$. By \zcref{cor:frame-riesz-stability}, there exists $t < 1$  such that $\pi_t(\Lambda) \eta$ is a Riesz sequence in $\Hrho$. The necessary density conditions for Riesz sequences (see \zcref{prop:necessary_density_conditions}) then give that $D^+(\Lambda) \leq d_{\pi_t}$, and thus
    \[ D^+(\Lambda) \leq d_{\pi_t} = t^d d_\pi < d_\pi , \]
    completing the proof.
\end{proof}

\section{Starting from projective representations} \label{sec:projective}

In this section we prove a version of the main theorem starting with a projective representation that is square-integrable in the strict sense rather than a genuine representation that is square-integrable modulo the center. This follows from the correspondence between such square-integrable representations, see, e.g., \cite[Appendix A]{papageorgiou2025counting}.

Throughout this section, we let $G$ be a nontrivial connected, simply connected, nilpotent Lie group equipped with a left Haar measure $\mu_G$. A projective unitary representation $(\pi, \Hpi)$ of $G$ is a strongly continuous map $\pi : G \to \mathcal{U}(\Hpi)$ satisfying $\pi(e_G) = I_{\Hpi}$ and
\[ 
\pi(x) \pi(y) = \sigma(x,y) \pi(x y), \quad x,y \in G 
\]
for a $2$-cocycle $\sigma : G \times G \to \mathbb{T}$. For simplicity, we assume that $\sigma$ is smooth. In this case, we denote the space of smooth vectors by $\Hpi^{\infty}$, that is, the space of all $f \in \Hpi$ such that $x \mapsto \pi(x) f$ is smooth on $G$. We assume that $(\pi, \Hpi)$ is irreducible, i.e., the only closed $\pi$-invariant subspaces are $\{0\}$ and $\Hpi$. An irreducible projective representation $(\pi, \Hpi)$ is called \emph{square-integrable} if there exists $f \in \Hpi \setminus \{0\}$ such that $\langle f, \pi(\cdot) f \rangle \in L^2 (G)$. In this case, there exists $d_{\pi} > 0$, called the \emph{formal degree}, such that
\[
\int_G |\langle f_1, \pi(x) f_2 \rangle |^2 \; d\mu_G (x) = d_{\pi}^{-1} \| f_1 \|^2 \| f_2 \|^2
\]
for all $f_1, f_2 \in \Hpi$.

We will deduce the following theorem from \zcref{thm:main_intro}.

\begin{theorem}
Let $(\pi, \Hpi)$ be an irreducible, square-integrable projective representation of $G$ of formal degree $d_{\pi} > 0$.
Let $\eta \in \Hpi^\infty$ and let $\Lambda \subseteq G$ be countable. The following hold:
\begin{enumerate}
    \item[(i)] If $\pi(\Lambda)\eta$ is a frame for $\Hrho$, then $D^-(\Lambda) > d_\pi$.
    \item[(ii)] If $\pi(\Lambda)\eta$ is a Riesz sequence in $\Hrho$, then $D^+(\Lambda) < d_\pi$.
\end{enumerate}
\end{theorem}
\begin{proof}
Let $\sigma$ be the smooth $2$-cocycle of $(\pi, \Hpi)$ and consider the group $G_\sigma:=G\times\mathbb T$ with multiplication
$
    (x,u)(y,v)=(xy,uv\sigma(x,y)).
$
Denote by $N$ its universal covering group, which is connected, simply connected,
and nilpotent. In additon, we have that $N=G\times\mathbb R$ as a manifold, with covering
map $p(x,r)=(x,e^{2\pi ir})$. The unitary operators
\[
    \rho(x,r) =e^{2\pi ir}\pi(x), \quad (x,r) \in N,
\]
define an irreducible unitary representation of $N$ on the Hilbert space $\Hpi$. 

We claim that the center $Z (N) = \{e_G\}\times \R$.
For this, note first that, for fixed $z\in \{e_G \} \times \R$, the map
$n\mapsto nzn^{-1}z^{-1}$ is continuous on $N$ and its range is contained in $\ker(p)$. Hence, as $N$ is connected, it follows that $n\mapsto nzn^{-1}z^{-1}$ is contant, which implies that
$\{e_G \} \times \R \subseteq Z(N)$.
Conversely, if $(x,r)\in Z(N)$, then Schur's lemma implies that
$\rho(x,r) \in \mathbb{T} I_{\Hpi}$, and hence $\pi(x) \in \mathbb{T} I_{\Hpi}$. Since the projective kernel $P_{\pi} := \{ x \in G : \pi(x) \in \mathbb{T} I_{\Hpi} \}$ is compact by \cite[Lemma 2.3]{enstad2022sufficient}, it must be trivial, and thus $x = e_G$. 

Since $N / Z(N) \cong G$ by the above, using the smooth section $s : G \to N, x \mapsto (x,0)$, we obtain
$\rho(s(x))=\pi(x)$. In particular, it follows immediately that $\rho$ is square-integrable
modulo $Z(N)$ with $d_\rho=d_\pi$ and that the spaces of smooth vectors of $\pi$ and $\rho$ agree. The claims follow therefore directly from 
\zcref{thm:main_intro}.
\end{proof}

\appendix

\section{Spectral invariance}\label{sec:spectral-invariance}

Let $\Lambda \subseteq G$ be relatively separated in a connected simply connected nilpotent Lie group $G$. Given a bounded linear operator $A$ on $\ell^2(\Lambda)$, we write $(A_{\lambda,\lambda'})_{\lambda,\lambda' \in \Lambda}$ for its representation with respect to the standard basis of $\ell^2(\Lambda)$, that is,
\[ A_{\lambda,\lambda'} = \langle A e_{\lambda'},  e_{\lambda} \rangle, \quad \lambda, \lambda' \in \Lambda. \]

As in \zcref{sec:coherent}, let $w_s(x) = (1+d(e,x))^s$, $s > 0$, denote the weights coming from a left-invariant Carnot-Carath\'eodory metric $d$ on $G$. We consider the following algebra of infinite matrices over the index set $\Lambda$:
\[ \mathcal{J}(\Lambda) = \Big\{ A \in \mathcal{B}(\ell^2(\Lambda)) \; : \; \forall s > 0 \; \sup_{\lambda,\lambda' \in \Lambda}|A_{\lambda,\lambda'}| w_s(\lambda^{-1}\lambda') < \infty \Big\} . \]
The matrix algebra $\mathcal{J}(\Lambda)$ is often referred to as the \emph{Jaffard matrix algebra} \cite{jaffard-1990}.

\begin{theorem}\label{thm:spectral-invariance}
If a self-adjoint operator with closed range belongs to $\mathcal{J}(\Lambda)$, then its Moore--Penrose inverse also belongs to $\mathcal{J}(\Lambda)$.
\end{theorem}

\begin{proof}
We apply \cite[Theorem 5.1]{sun2007} and follow the notation there. Our setting is the discrete set $ \Lambda$, the metric $d|_{\Lambda \times \Lambda}$ and the counting measure $\mu_c$ on $\Lambda$. The equation \eqref{eq:rel-sep-pol-growth} gives the necessary polynomial growth property of $d|_{\Lambda \times \Lambda}$ and the counting measure; see  \cite[Equation (2.1)]{sun2007}. For $s > 0$ we consider the weight in two variables $v_s(\lambda,\lambda') = w_s(\lambda^{-1}\lambda')$ for $\lambda,\lambda' \in \Lambda$. We are concerned with the corresponding Schur class defined in  \cite[Equation (2.13)]{sun2007} with $p= \infty$:
\[ \mathcal{A}_{\infty,v_s}(\Lambda,d,\mu_c) = \Big\{ A \in \C^{\Lambda \times \Lambda}: \sup_{\lambda \in \Lambda} \| A_{\lambda, \cdot}v_s(\lambda, \cdot) \|_\infty + \sup_{\lambda' \in \Lambda} \| A_{\cdot, \lambda'} v_s(\cdot, \lambda') \|_\infty < \infty \Big\} . \] 
These sets are nested: If $s > t > 0$, then $\mathcal{A}_{\infty,v_s}(\Lambda,d,\mu_c) \subseteq \mathcal{A}_{\infty,v_t}(\Lambda,d,\mu_c)$. Notice that in our case of $p=\infty$, the condition collapses to $2 \sup_{\lambda,\lambda' \in \Lambda} |A_{\lambda,\lambda'} | v_s(\lambda,\lambda') < \infty$, from which it follows that
\[ \mathcal{J}(\Lambda) = \bigcap_{s > 0} \mathcal{A}_{\infty,v_s}(\Lambda,d,\mu_c) .\]
By \cite[Example A.2]{sun2007}, the weight $v_s$ is $(\infty,2)$-admissible provided $s > d(G)+1$. Thus, all hypotheses are satisfied to apply \cite[Theorem 5.1]{sun2007}. 

Let $A$ be self-adjoint with closed range, and suppose that $A \in \mathcal{J}(\Lambda)$. Letting $H = (\mathrm{Ker} A)^{\perp}$ and letting $P$ be the orthogonal projection of $\ell^2(\Lambda)$ onto $H$, the closed range of $A$ implies that there is a constant $D > 0$ such that
\[ \| A c \| \geq D \| c \|, \quad c \in H. \]
Thus \cite[Theorem 5.1]{sun2007} gives that the Moore--Penrose inverse $A^\dagger$ is in $\mathcal{A}_{\infty,v_s}(\Lambda,d,\mu_c)$ for all $s > d(G)+1$. It follows that $A^\dagger \in \mathcal{J}(\Lambda)$, which finishes the proof.
\end{proof}

\section{Polynomial phase multiplier on Schwartz space}

In the proof of \zcref{thm:continuity-in-t}, we use the following lemma on polynomial phase multipliers. The fact that the Schwartz space is invariant under such operators was already shown in \cite[Lemma 3.1]{howe1977on}. 

\begin{lemma} \label{lem:howe}
Let $\alpha : \R^d \to \R$ be a polynomial. If $\mathcal{B} \subseteq \mathcal{S}(\R^d)$ is bounded and $t_0 > 0$, then
\[
e^{2\pi i t \alpha } f \to e^{2\pi i t_0 \alpha } f \quad \text{as} \quad t \to t_0
\]
in $\mathcal{S}(\R^d)$, uniformly for $f \in \mathcal{B}$.
\end{lemma}
\begin{proof}
Fix a compact interval $I \subseteq (0, \infty)$ with $t_0$ in its interior and let $t \in I$.
For fixed $t \in I$, we write $M_t f := e^{2\pi i t \alpha } f$. The continuity of each $M_t$ on $\mathcal{S}(\R^d)$ follows from \cite[Lemma 3.1]{howe1977on}.

For \(L\in\mathbb N_0\) and a multi-index \(\beta \in \N_0^d\), write
$
    p_{L,\beta}(f)
    =
    \sup_{x\in\R^d}
    (1+|x|)^L\lvert\partial^\beta f(x)\rvert
$
for the usual seminorms on $\mathcal{S}(\R^d)$.
For \(u\in I\), define
$
    a_u(x)
    :=
    2\pi i\alpha(x)e^{2\pi iu\alpha(x)}.
$
Then, given a multi-index \(\gamma \in \N_0^d\), it follows by induction that there exists a polynomial
\(P_{\gamma,u}\) in \(x\) such that
\[
    \partial^\gamma a_u(x)
    =
    P_{\gamma,u}(x)e^{2\pi iu\alpha(x)}.
\]
The coefficients of \(P_{\gamma,u}\) depend polynomially on \(u\).
Since \(I\) is compact, for every \(\gamma  \in \N_0^d\), there exists
\(C_\gamma>0\) and \(m_\gamma\in\mathbb N_0\) such that
\[
    \lvert \partial^\gamma a_u(x) \rvert
    \leq
    C_\gamma(1+|x|)^{m_\gamma},
    \qquad x\in\R^d, \; u\in I.
\]
Set $m := \max_{\gamma \leq \beta} m_{\gamma} $.
This estimate, together with Leibniz' rule, thus yields a constant $C>0$ such that
\begin{align} \label{eq:uniform_seminorm}
    p_{L,\beta}(a_u f)
    \leq
    C\sum_{\gamma\leq\beta}
    p_{L+m,\beta-\gamma}(f),
    \qquad u\in I.
\end{align}

Lastly, note that the fundamental theorem of calculus gives
\[
    M_t f -M_{t_0} f
    =
    \int_{t_0}^t
    2\pi i\alpha e^{2\pi iu\alpha} f\,du
    =
    \int_{t_0}^t a_u f \,du.
\]
The uniform estimate \eqref{eq:uniform_seminorm}, together with differentiation under the integral, therefore yields 
\[
\begin{aligned}
    p_{L,\beta}\bigl((M_t-M_{t_0})f \bigr)
    \leq
    \int_{\min\{t_0,t\}}^{\max\{t_0,t\}}
    p_{L,\beta}(a_u f)\,du               \leq
    C|t-t_0|
    \sum_{\gamma\leq\beta}
    p_{L+m,\beta-\gamma}(f).
\end{aligned}
\] 
Using that $\mathcal{B} \subseteq \mathcal{S}(\R^d)$ is bounded, it follows therefore that
\[
    \sup_{f \in B}
    p_{L,\beta}\bigl((M_t-M_{t_0}) f\bigr)
    \to 0
    \qquad \text{as} \quad t\to t_0.
\]
Since $L \in \mathbb{N}_0$ and $\beta \in \mathbb{N}_0^d$ were chosen arbitrarily, this completes the proof.
\end{proof}

\section*{Tool and computational resource disclosure}
Large language models (LLMs), specifically ChatGPT 5.6 Sol and ChatGPT 6 Astra, played an important role in the development of this work. Over the preceding years, the authors had made several attempts to prove \zcref{thm:main_intro}, including attempts to establish the stability of coherent frames and Riesz sequences under deformations of their point sets or the associated representations. The LLMs were first used to try to prove \zcref{thm:main_intro} using deformations of point sets, but these attempts did not yield suitable deformations beyond the setting of homogeneous groups. The LLMs were then used to develop a proof based on deforming the representations. This led to a nonexistence result for smooth Riesz bases, which we then realized could be strengthened to strict density inequalities for general coherent Riesz sequences. LLMs were subsequently also used to obtain a dual statement for coherent frames.

The initial proof of \zcref{thm:continuity-in-t} suggested by the LLMs was not based on induction and relied on a more brute-force approach than the current proof. To obtain a more conceptual proof, we examined the various known proofs of \zcref{thm:schwartz_matrix} with an eye to adapting them to the continuity claim of \zcref{thm:continuity-in-t}. We eventually used the LLMs to provide a proof of \zcref{thm:continuity-in-t} based on the ideas underlying the proof of \zcref{thm:schwartz_matrix} in \cite{corwin1996lp}. We then significantly rewrote the resulting argument to obtain the current proof. The proofs of the deformation results for frames and Riesz sequences in \zcref{sec:coherent} are based on initial proof strategies suggested by the LLMs, which we partly simplified using more elementary arguments or suitable references.

The authors have independently verified, validated, and rewritten all parts of the paper influenced by LLM-generated material and take full responsibility for the mathematical content
of the paper.

\section*{Acknowledgements}
For J.v.V., this research was funded in whole or in part by the Austrian Science Fund
(FWF): 10.55776/PAT2545623. For open access purposes, the author has applied a CC
BY public copyright license to any author-accepted manuscript version arising from this
submission.

\bibliographystyle{abbrv}
\bibliography{bib}

\end{document}